\documentclass[12pt]{article}
\usepackage{amsfonts,amssymb,amsthm,amsmath,cite,bbm}
\usepackage{mathptmx}
\usepackage{enumerate}
\usepackage{enumitem}
\usepackage[pagebackref=true]{hyperref}
\usepackage{cite}
\usepackage{comment} 
\usepackage{geometry}
\hypersetup{
	colorlinks,%
	citecolor=blue,%
	filecolor=black,%
	urlcolor=black
}

\newcommand\xqed[1]{%
  \leavevmode\unskip\penalty9999 \hbox{}\nobreak\hfill
  \quad\hbox{#1}}
\newcommand\dssymb{\xqed{\small //}}

\newtheorem{theorem}{Theorem}[section]
\newtheorem{lemma}[theorem]{Lemma}
\newtheorem{proposition}[theorem]{Proposition}
\newtheorem{corollary}[theorem]{Corollary}
\theoremstyle{remark}
\newtheorem{remark}[theorem]{Remark}
\numberwithin{equation}{section}

\newcommand\blfootnote[1]{%
  \begingroup
  \renewcommand\thefootnote{}\footnote{#1}%
  \addtocounter{footnote}{-1}%
  \endgroup
}

\newcommand{\N}{{\mathbb N}} 
\newcommand{\R}{{\mathbb R}}
\newcommand{\Rn}{{\mathbb R}^n}
\newcommand{\RN}{{\mathbb R}^{n+1}}
\newcommand{\s}{\mathbb{S}}
\newcommand{\sn}{\mathbb{S}^{n-1}}
\newcommand{\sN}{\mathbb{S}^{n}}
\newcommand{\Bn}{B^n}

\newcommand{\K}{{\mathcal K}}
\newcommand{\Kn}{{\mathcal K}^n}
\newcommand{\KN}{{\mathcal K}^{n+1}}
\newcommand{\KNH}{{\mathcal K}_{\langle H\rangle}^{n+1}}

\newcommand{\hm}{\mathcal H}

\newcommand{\fconvs}{{\mathrm{Conv}_{\mathrm{sc}}(\R^n)}} 
\newcommand{\fconvf}{{\mathrm{Conv}(\R^n; \R)}}
\newcommand{\fconvcd}{{\mathrm{Conv}_{{\mathrm cd}}(\R^n)}} 
\newcommand{\fconvz}{{\mathrm{Conv}_{{\leq 0}}(\R^n)}} 
\newcommand{\Cr}{C_\mathrm{rec}(\Rn)}
\newcommand{\proj}{\operatorname{proj}}
\newcommand{\gnom}{\operatorname{gno}} 
\newcommand{\refl}{\operatorname{ref}} 

\newcommand{\infconv}{\mathbin{\Box}} 
\newcommand{\sq}{\mathbin{\vcenter{\hbox{\rule{.3ex}{.3ex}}}}} 

\newcommand{\MA}{\mathrm{MA}} 
\newcommand{\MAp}{{\mathrm{MA}^{\!*}}} 
\newcommand{\Sa}{\Lambda^{\!*}} 
\newcommand{\Sap}{\Lambda} 

\DeclareMathOperator{\oZ}{\operatorname{Z}}
\newcommand{\oVb}{\overline{\operatorname{V}}} 
\newcommand{\oZZb}[2]{\overline{\operatorname{V}}_{#1,#2}} 

\newcommand{\Hess}{{\operatorname{D}}^2}

\newcommand{\Grass}[2]{\operatorname{G}(#2,#1)}

\renewcommand{\d}{\,\mathrm{d}}
\newcommand{\ind}{{\mathbf{I}}}
\newcommand{\interior}{\operatorname{int}} 
\newcommand{\bd}{\operatorname{bd}} 
\newcommand{\dom}{\operatorname{dom}} 
\newcommand{\epi}{\operatorname{epi}} 
\newcommand{\supp}{\operatorname{supp}} 

\title{Inequalities and Counterexamples for\\Functional Intrinsic Volumes and Beyond}
\author{Fabian Mussnig and Jacopo Ulivelli}

\date{}

\begin{document}
\maketitle

\begin{abstract}
We show that analytic analogs of Brunn--Minkowski-type inequalities fail for functional intrinsic volumes on convex functions. This is demonstrated both through counterexamples and by connecting the problem to results of Colesanti, Hug, and Saor\'in G\'omez. By restricting to a smaller set of admissible functions, we then introduce a family of variational functionals and establish Wulff-type inequalities for these quantities. In addition, we derive inequalities for the corresponding family of mixed functionals, thereby generalizing an earlier Alexandrov--Fenchel-type inequality by Klartag and recovering a special case of a recent P\'olya--Szeg\H{o}-type inequality by Bianchi, Cianchi, and Gronchi.

\end{abstract}

{
    \hypersetup{linkcolor=black}
    \small
    \tableofcontents
}

\goodbreak

\normalsize

\section{Introduction}
\subsection{Geometric Inequalities}
\label{se:geometric_ineq}
Among the cornerstones of convex geometry are geometric inequalities. The central and probably most elegant result in this area is the \textit{Brunn--Minkowski inequality}, which states that
\begin{equation}
\label{eq:bm}
V_n(K+L)^{\frac 1n} \geq V_n(K)^{\frac 1n} + V_n(L)^{\frac 1n},
\end{equation}
for all convex bodies $K$ and $L$, i.e.\ non-empty, compact, convex subsets of $\Rn$, the set of which we denote by $\Kn$. Here, $V_n$ is the usual \textit{volume} or \textit{Lebesgue measure}, and
\[
K+L=\{x+y: x\in K, y\in L\}
\]
is the \textit{Minkowski sum} of $K$ and $L$. Equality holds in \eqref{eq:bm} if and only if $K$ and $L$ are homothetic or lie in parallel hyperplanes. An excellent exposition of the importance of this fundamental result, as well as an overview of its many significant applications, can be found in \cite{Gardner_BM}. Generalizing \eqref{eq:bm}, we have
\begin{equation}
\label{eq:bm_general}
V_j((1-\lambda) K + \lambda L)^{\frac 1j} \geq (1-\lambda) V_j(K)^{\frac 1j} + \lambda V_j(L)^{\frac 1i}
\end{equation}
for every $K,L\in\Kn$, $0\leq \lambda \leq 1$, and $0\leq j \leq n$. Here, $V_j\colon\Kn\to [0,\infty)$, $0\leq j\leq n$, is the $j$th intrinsic volume, which arises from the \textit{Steiner formula}
\begin{equation}
\label{eq:steiner}
V_n(K+r\,B^n)=\sum_{j=0}^n r^{n-j} \kappa_{n-j} V_j(K)
\end{equation}
for $K\in\Kn$ and $r>0$, where for $k\in\N$ we write $B^k$ for the Euclidean ball in $\R^k$ and denote by $\kappa_{k}$ its $k$-dimensional volume (with the convention $\kappa_0=1$).

\medskip

One of the many far-reaching consequences of the Brunn--Minkowski inequality is the classical Euclidean \textit{isoperimetric inequality}, which states that Euclidean balls minimize surface area among all convex bodies with given volume. More generally,
\begin{equation}
\label{eq:iso}
\frac{V_1(K)}{V_1(B^n)} \geq \left(\frac{V_2(K)}{V_2(B^n)} \right)^{\frac 12} \geq \cdots \geq \left(\frac{V_{n-1}(K)}{V_{n-1}(B^n)} \right)^{\frac{1}{n-1}} \geq \left(\frac{V_n(K)}{V_n(B^n)} \right)^{\frac 1n}
\end{equation}
for every $K\in\Kn$, where the final inequality is the isoperimetric inequality.

\medskip

Even more powerful inequalities are possible once we consider mixed functionals. The \textit{mixed volume} $V\colon (\Kn)^n\to [0,\infty)$ is the unique symmetric functional such that
\begin{equation}
\label{eq:mixed_vol}
V_n(\lambda_1 K_1 +\cdots + \lambda_m K_m)=\sum_{i_1,\ldots,i_n=1}^m \lambda_{i_1}\cdots\lambda_{i_n} V(K_{i_1},\ldots,K_{i_n})
\end{equation}
for every $m\in\N$, $K_1,\ldots,K_m\in\Kn$, and $\lambda_1,\ldots,\lambda_m\geq 0$. The key result for these objects is the \textit{Alexandrov--Fenchel inequality}, which states that
\begin{equation}
\label{eq:AF}
V(K_1,\ldots,K_n)^2 \geq V(K_1,K_1,K_3,\ldots,K_n) V(K_2,K_2,K_3,\ldots,K_n)
\end{equation}
for $K_1,\ldots,K_n\in\Kn$. While equality cases are known for special cases of $K_1,\ldots,K_n$ (we mention the recent breakthrough for polytopes \cite{shenfeld_van_handel_acta} and its extension to more general polyoids \cite{Hug_Reichert_AF}), the general case is still open. Furthermore, we remark that more recent achievements from the theory of valuations allow for even stronger inequalities than those mentioned so far \cite{Bernig_Kotrbaty_Wannerer,Kotrbaty_Wannerer_GAFA}.

\subsection{Functional Intrinsic Volumes}
It is a well-studied fact that many geometric inequalities have analytic counterparts. A classical example is the correspondence between the Euclidean isoperimetric inequality and the Sobolev inequality. We refer to \cite[Chapter 9]{Artstein_Giannopoulos_Milman} and \cite[Section 10.15]{Schneider_CB} for overviews and mention \cite{Haddad_Ludwig_JDG} as a recent example. The focus of this article is on inequalities that concern functionals on various spaces of convex functions. Of particular interest are the recently introduced \textit{functional intrinsic volumes} \cite{Colesanti-Ludwig-Mussnig-5} on the space
\[
\fconvs = \left\{u\colon \Rn \to (-\infty,\infty] : u \text{ is l.s.c., convex, } u\not\equiv \infty\text{, } \lim\nolimits_{|x|\to\infty} \tfrac{u(x)}{|x|}=\infty\right\}
\]
of proper, lower semicontinuous, super-coercive, convex functions. Note that a function $u$ on $\Rn$ is an element of $\fconvs$ if and only if its \textit{convex conjugate} $u^*$ is an element of
\[
\fconvf = \{v\colon \Rn\to \R : v \text{ is convex}\},
\]
and it is, therefore, often beneficial to consider dual results on both of these spaces. To define functional analogs of the intrinsic volumes on $\fconvs$, consider a density function $\alpha\in C_c([0,\infty))$, that is, $\alpha$ is continuous with compact support. We now set
\begin{equation}
\label{eq:ozzb_n_alpha}
\oZZb{n}{\alpha}(u)=\int_{\dom(u)} \alpha(|\nabla u(x)|)\d x
\end{equation}
for $u\in\fconvs$, where we write $\dom(u)=\{x\in\Rn : u(x)<\infty\}$ for the domain of $u$ and remark that convex functions are differentiable almost everywhere on the interior of their domains. Let us point out that in the dual setting of functions in $\fconvf$, the functional dual to \eqref{eq:ozzb_n_alpha}, given by $v\mapsto \oZZb{n}{\alpha}(v^*)$, is an integral with respect to the \textit{Monge--Amp\`ere measure} of $v$ (see Section~\ref{se:convex_functions} for details; see also \eqref{eq:int_nabla_v_conj_MA}).

\medskip

An immediate consequence of \eqref{eq:ozzb_n_alpha} is
\begin{equation}
\label{eq:ozz_on_ind}
\oZZb{n}{\alpha}(\ind_K)=\alpha(0)V_n(K)
\end{equation}
for every $K\in\Kn$, where
\[
\ind_K(x)=\begin{cases}
0\quad &\text{if } x\in K,\\
\infty\quad &\text{if } x\not\in K,
\end{cases}
\]
is the \textit{(convex) indicator function} of $K$. Similar to the Steiner formula \eqref{eq:steiner} we now obtain a family of functionals $\oZZb{j}{\alpha}\colon\fconvs\to\R$ through the relation
\begin{equation}
\label{eq:ozzb_steiner}
\oZZb{n}{\alpha}(u\infconv \ind_{r\,B^n}) = \sum_{j=0}^n r^{n-j} \kappa_{n-j} \oZZb{j}{\alpha}(u)
\end{equation}
for $u\in\fconvs$ and $r>0$, where $\infconv$ denotes \textit{infimal convolution} of functions (we refer to Section~\ref{se:convex_functions} for a definition). Since $\ind_K+\ind_{r\,B^n} = \ind_{K+r\,B^n}$, this immediately implies that
\[
\oZZb{j}{\alpha}(\ind_K)=\alpha(0)V_j(K)
\]
for $K\in\Kn$ and $0\leq j\leq n$, which shows that intrinsic volumes of convex bodies can be retrieved from functional intrinsic volumes.

\medskip

The operators $\oZZb{j}{\alpha}$ share many properties with the classical intrinsic volumes, and a full characterization of the former as continuous valuations with additional invariance properties, similar to Hadwiger's characterization of the intrinsic volumes, was obtained by the first author together with Colesanti and Ludwig in \cite{Colesanti-Ludwig-Mussnig-5}. Further proofs of this result were given in \cite{Colesanti-Ludwig-Mussnig-8,Knoerr_singular}. For additional properties of functional intrinsic volumes, integral geometry, and the corresponding theory of valuations on convex functions, we refer to \cite{Colesanti-Ludwig-Mussnig-4,Colesanti-Ludwig-Mussnig-6,Colesanti-Ludwig-Mussnig-7,Knoerr_support,Knoerr_unitarily,Knoerr_unitarily_decomp,Knoerr_smooth,Knoerr_Ulivelli,Knoerr_Ulivelli_2,Hug_Mussnig_Ulivelli,Hug_Mussnig_Ulivelli_2}. The study of functional intrinsic volumes also attracted considerable interest due to its connections with \textit{zonal valuations} on convex bodies, which in turn are related to rotation covariant \textit{Minkowski valuations} \cite{schuster_wannerer_JEMS}. We refer to the very recent works \cite{Brauner_Hofstaetter_Ortega-Moreno,Knoerr_zonal}, where ideas from the analytic setting were used to successfully establish a full classification of zonal valuations on convex bodies.

\subsection{Counterexamples}
While many properties for functional intrinsic volumes have been proved, which show that they behave similarly to their classical counterparts, the question of elementary inequalities for the operators $\oZZb{j}{\alpha}$ is up until now open. In particular, since the classical intrinsic volumes are retrieved from their functional versions, it would be natural to anticipate that analogs of the Brunn--Minkowski inequality \eqref{eq:bm} and the isoperimetric inequalities \eqref{eq:iso} also hold for functional intrinsic volumes. However, it turns out that this is not the case. In Section~\ref{se:bm_ineq} we show that apart from the case $j=1$, an inequality of the form
\[
\big(\oZZb{j}{\alpha}(u_1\infconv u_2)\big)^{\frac 1j} \geq \big(\oZZb{j}{\alpha}(u_1) \big)^{\frac 1j} + \big(\oZZb{j}{\alpha}(u_2)\big)^{\frac 1j}
\]
with $u_1,u_2\in\fconvs$, $2\leq j\leq n$, and non-negative $\alpha\in C_c([0,\infty))$, does not hold in general. While Proposition~\ref{prop:BM_conv} presents a counterexample, we will also give a deeper explanation of this phenomenon by connecting this question to a result for Brunn--Minkowski-type inequalities on convex bodies by Colesanti, Hug, and Saor\'in G\'omez \cite{Colesanti-Hug-Saorin_Gomez_JGA}.

\medskip

The results above suggest that it may be overly optimistic to expect classical geometric inequalities to transfer directly to the functional setting. With regard to the isoperimetric inequality, however, we show in Proposition~\ref{prop:iso_counter} that an isoperimetric principle for functional intrinsic volumes does not even hold in the form of an extreme value problem.

\subsection{Wulff-type Inequalties}
Having dealt with negative results in Section~\ref{se:counterexample}, we establish new inequalities in Section~\ref{se:wulff_type} that concern more general functionals from the family of \textit{Hessian valuations} \cite{Colesanti-Ludwig-Mussnig-3}, which include functional intrinsic volumes as a special case. For this, we restrict to convex functions from the smaller space
\[
\fconvz = \{u\in\fconvs : \dom(u) \text{ is compact, } \max\nolimits_{x\in\dom(u)} u(x)\leq 0\},
\]
which is dual to the space
\begin{equation}
\label{eq:fonvz_conj}
\fconvz^*=\{v\in\fconvf : 0 \leq v-h_K \leq C \text{ for some } K\in\Kn \text{ and } C\geq 0\}
\end{equation}
by convex conjugation, where $h_K(x)=\max_{y\in K} \langle x,y\rangle$ is the \textit{support function} of $K\in\Kn$. For such a function $u\in\fconvz$ we now consider the functional $\oVb_{n+1}\colon \fconvz\to\R$, given by
\begin{equation}
\label{eq:def_ovb_n+1}
\oVb_{n+1}(u)=\int_{\dom(u)}|u(x)|\d x
\end{equation}
for $u\in\fconvz$.

Next, let $\rho_\varphi\colon \Rn \to [-\infty,\infty]$ denote the \textit{recession function} of $\varphi\in C(\Rn)$, which, presuming that the limit exists, is defined by
\begin{equation}
\label{eq:rho_varphi}
\rho_\varphi(z)=\lim\nolimits_{t\to\infty} \frac{\varphi(t z)}{t}
\end{equation}
for $z\in\sn$ and which we extend (positively) $1$-homogeneously to $\Rn$. In addition, we set
\begin{equation}
\label{eq:def_Cr}
\Cr=\{\varphi \in C(\Rn): \rho_\varphi \text{ is continuous and } |\rho_\varphi-\varphi| \text{ is bounded on } \Rn\}.
\end{equation}
Here, the condition of boundedness serves to ensure that the convex conjugate $\varphi^*$ has compact domain (cf.\ \cite[Remark 3.6]{Ulivelli_Wulff}). Note that $\fconvz^*\subset \Cr$. For $\varphi\in\Cr$ we define $\oZ_\varphi\colon \fconvz\to\R$ as
\begin{equation}
\label{eq:oz_varphi_intro}
\oZ_\varphi(u)=\int_{\dom(u)} \varphi(\nabla u(x))\d x + \int_{\bd(\dom(u))} |u(x)| \rho_{\varphi}(\nu_{\dom(u)}(x))\d\hm^{n-1}(x)
\end{equation}
whenever $u\in\fconvz$ has full-dimensional domain. Here, $\bd(\dom(u))$ denotes the \textit{boundary} of $\dom(u)$ and $\hm^{n-1}$ is the $(n-1)$-dimensional \textit{Hausdorff measure}. In addition, $\nu_{\dom(u)}(x)\in\sn$ is the outer unit normal to $\dom(u)$ at $x\in \bd(\dom(u))$, which is well-defined $\hm^{n-1}$-a.e.\ on $\bd(\dom(u))$. Note that our conditions on $\varphi$, together with the fact that elements of $\fconvz$ have compact domain, guarantee that \eqref{eq:oz_varphi_intro} is finite. For a precise definition of $\oZ_\varphi$, using appropriate measures, we refer to Section~\ref{suse:bodies_fcts}.

\medskip

Let us point out that $\oZ_\varphi$ generalizes $\oZZb{n}{\alpha}$ on $\fconvz$. Indeed, for $\varphi(x)=\alpha(|x|)$, $x\in\Rn$, with some $\alpha\in C_c([0,\infty))$, we have $\rho_{\varphi}\equiv0$ and therefore $\oZ_{\varphi}=\oZZb{n}{\alpha}$. In addition, we will show in Proposition~\ref{prop:varphi_u_conj} that \eqref{eq:oz_varphi_intro} also generalizes the functional $\oVb_{n+1}$ through the relation
\[
\frac{1}{n+1}\oZ_{u^*}(u)=\oVb_{n+1}(u)
\]
for $u\in\fconvz$. In fact, the true nature of $\oZ_\varphi$ is that of a surface area integral, and under mild assumptions on $u\in\fconvz$ and $\varphi\in\Cr$ we will prove the variational formula
\[
\oZ_{\varphi}(u)=
\frac{\d}{\d t}\bigg\vert_{t=0^+} \oVb_{n+1}((u^*+t\varphi)^*)
\]
in Proposition~\ref{prop:variational_formula}. See also \cite{Huang_Liu_Xi_Zhao_dual_LC,Kryvonos_Langharst,Ulivelli_Wulff} for related results.

\medskip

Our main result of Section~\ref{se:wulff_type} is the following inequality, where we use the notation $\{u\leq t\}=\{x\in\Rn : u(x)\leq t\}$ for the \textit{level set} of a convex function $u$ at height $t\in\R$.
\begin{theorem}
\label{thm:wulff_ineq_fconvz}
If $\varphi\in\Cr$ is such that  
\begin{equation*}
c\leq \frac{\varphi(x)}{\sqrt{1+|x|^2}}
\end{equation*}
for every $x\in\Rn$ and some $c>0$, then

\begin{equation}
\label{eq:wulff_fct}
\left(\frac{\oZ_\varphi(u)}{\oZ_\varphi(u_\varphi)}\right)^{\frac{1}{n}}\geq \left(\frac{\oVb_{n+1}(u)}{\oVb_{n+1}(u_\varphi)}\right)^{\frac{1}{n+1}}
\end{equation}
for every $u\in\fconvz$, where $u_\varphi=\varphi^*+\ind_{\{\varphi^*\leq 0\}}$. If $\oVb_{n+1}(u)>0$, then equality holds in \eqref{eq:wulff_fct} if and only if $u(x)=\lambda\, u_\varphi(\frac{x-x_o}{\lambda})$, $x\in\Rn$, for some $\lambda>0$ and $x_o\in\Rn$. In addition, if $\rho_\varphi \leq \varphi$ on $\Rn$, then $u_\varphi=\varphi^*$.
\end{theorem}

Our approach to proving \eqref{eq:wulff_fct}, which in part we also use in Section~\ref{se:counterexample}, is to associate with each convex function $u\in\fconvz$ a convex body in $\KN$ and to derive inequalities for $u$ from known results for convex bodies. In the case of Theorem~\ref{thm:wulff_ineq_fconvz}, the underlying inequality is \textit{Wullf's theorem} (see Theorem~\ref{thm:wulff} below). This method is, of course, not new. We refer to \cite{Artstein_Klartag_Milman_marginals,Artstein_Klartag_Schuett_Werner,Caglar_et_al,Caglar_Werner,Klartag_05,Klartag_07} for some skillful and insightful demonstrations of this strategy and also mention \cite{Colesanti-Ludwig-Mussnig-8,Hug_Mussnig_Ulivelli_2,Knoerr_support,Knoerr_zonal,Knoerr_Ulivelli}, where this technique was used in the context of valuations.

\medskip

Let us furthermore point out that Theorem~\ref{thm:wulff_ineq_fconvz} can be easily formulated in terms of the conjugate functions $u^*\in\fconvz^*$. Notably, assuming that all occurring expressions are integrable, we have
\begin{equation}
\label{eq:int_nabla_v_conj_MA}
v\mapsto \int_{\dom(v^*)} \varphi(\nabla v^*(x)) \d x = \int_{\Rn} \varphi( x) \d\MA(v;x)
\end{equation}
for $v\in\fconvf$ and $\varphi\colon \Rn\to\R$, where $\MA(v;\cdot)$ denotes the Monge--Amp\`ere measure associated with $v$. Under additional $C^2$ assumptions on $v$, this measure is absolutely continuous with respect to the Lebesgue measure with
\begin{equation}
\label{eq:MA_Hess}
\d\MA(v;x)=\det(\Hess v(x)) \d x,
\end{equation}
where we write $\det(\Hess v(x))$ for the determinant of the \textit{Hessian matrix} of $v$ at $x\in\Rn$. If we furthermore define $\Sap(v;\cdot)$ as the the push-forward of $\nu_{\dom(v^*)}$ under $|v^*(x)|\d\hm^{n-1}(x)$ restricted to $\bd(\dom(v^*))$ (see Section~\ref{suse:bodies_fcts} for precise definitions), then for the operator $v\mapsto \oZ_{\varphi}^*(v)=\oZ_{\varphi}(v^*)$, $\varphi\in\Cr$, we have
\[
\oZ_{\varphi}^*(v)=\int_{\Rn}\varphi(x)\d\MA(v;x) +\int_{\sn} \rho_{\varphi}(z)\d\Sap(v;z)
\]
for $v\in\fconvz$. Similarly, for $\oVb_{n+1}^*(v)=\oVb_{n+1}(v^*)$, Proposition~\ref{prop:varphi_u_conj} shows that
\[
\oVb_{n+1}^*(v)=\frac{1}{n+1}\oZ_{v}^*(v)
\]
for $v\in\fconvz^*$. Inequality \eqref{eq:wulff_fct} can now be rewitten as
\[
\left(\frac{\oZ_{\varphi}^*(v)}{\oZ_{\varphi}^*(v_{\varphi})} \right)^{\frac 1n} \geq \left( \frac{\oZ_{v}^*(v)}{\oZ_{v_{\varphi}}(v_\varphi) } \right)^{\frac{1}{n+1}},
\]
for $v\in\fconvz^*$, where $v_{\varphi}=\varphi^{**}$ in case $\rho_{\varphi}\leq\varphi$. Morever, if $\oZ_{v}^*(v)>0$, then equality holds if and only if $v(x)=\lambda\, v_{\varphi}(x)+\langle x_o,x\rangle$, $x\in\Rn$, for some $\lambda>0$ and $x_o\in\Rn$.

\subsection{Mixed Functionals and Shadow Systems}
In Section~\ref{se:mixed} we generalize $\oVb_{n+1}$ and look at the mixed functional $\oVb$, which arises from 
\[
\oVb_{n+1}\big((\lambda_1 \sq u_1) \infconv \cdots \infconv (\lambda_m \sq u_m)\big) = \sum_{i_0,\ldots,i_n=1}^m \lambda_{i_0}\cdots \lambda_{i_n} \oVb(u_{i_0},\ldots,u_{i_n})
\]
for $m\in\N$, $u_1,\ldots,u_m\in\fconvz$, and $\lambda_1,\ldots,\lambda_m\geq 0$. Here, $\lambda \sq u$ denotes \textit{epi-mul\-ti\-pli\-ca\-tion} of $u\in\fconvs$ with $\lambda\geq 0$, which is the canonical choice of multiplication with a scalar to be considered together with infimal convolution (see Section~\ref{se:convex_functions}).
Since $\oVb(u_0,\ldots,u_n)$ is, in fact, the mixed volume of some convex bodies associated with $u_0,\ldots,u_n$ (see \eqref{eq:def_ovb}), we immediately retrieve properties of $\oVb$ from known properties for mixed volumes. In particular, this results in an Alexandrov--Fenchel-type inequality on $\fconvz$. In order to give meaning to such results, we establish a representation of $\oVb$ in terms of mixed measures of convex functions in Proposition~\ref{prop:rep_ovb}. Our approach is essentially the same used by Klartag in \cite[Theorem 1.3]{Klartag_07}, who considered an equivalent setting of concave functions, although with additional regularity assumptions (see Remark~\ref{re:klartag_af} for details). As a special case of Corollary~\ref{cor:af_klartag}, we retrieve Klartag's inequality
\begin{align*}
\bigg(\int_{\Rn} u_0^*(x)\, D(\Hess u_1^*(x),&\ldots,\Hess u_n^*(x)) \d x\bigg)^2\\
&\geq \int_{\Rn} u_0^*(x)\,D(\Hess u_1^*(x),\Hess u_1^*(x),\Hess u_3^*(x),\ldots,\Hess u_n^*(x))\d x\\
&\quad \times \int_{\Rn} u_0^*(x)\,D(\Hess u_2^*(x),\Hess u_2^*(x),\Hess u_3^*(x),\ldots,\Hess u_n^*(x))\d x
\end{align*}
for sufficiently regular $u_0,\ldots,u_n\in\fconvz$, where $D$ denotes the \textit{mixed discriminant}, which takes $n$ symmetric $n\times n$ matrices as arguments. Let us furthermore remark that different functional analogs of mixed volumes on convex functions were previously considered, for example, in \cite{Milman_Rotem_JFA}.

\medskip

In Section~\ref{se:shadow_sytems}, we approach \textit{shadow systems} of convex functions to retrieve further inequalities for the functional $\oVb$. In particular, we obtain results for the \textit{Steiner symmetrals} $s_z u$ of a convex function $u$ in direction $z\in\sn$ and for the \textit{symmetric rearrangement} of $\bar{u}$ (we refer to Section~\ref{se:shadow_sytems} for detailed definitions). Our main result of this section is the following theorem, where we use the function space
\[
\fconvcd=\{u \in \fconvs: \dom(u) \text{ is compact} \},
\]
and where we write $\bd(K)$ for the boundary of a convex body $K$.

\begin{theorem}
\label{thm:int_w_nabla_u}
Let $w\in \fconvs$ and let $u\in\fconvcd$ be such that $u\vert_{\bd(\dom(u))} = c$ for some $c\in\R$. If
\begin{equation}
\label{eq:cond_int_w_nabla_u}
\int_{\dom(u)} w^*(\nabla u(x)) \d x < \infty,
\end{equation}
then
\begin{equation}
\label{eq:int_ineq_sz_bar}
\int_{\dom(u)} w^*(\nabla u(x)) \d x \geq \int_{\dom(s_z(u))} (s_z w)^*(\nabla (s_z u)(x))\d x \geq \int_{\dom(\bar{u})} \bar{w}^*(\nabla \bar{u}(x)) \d x
\end{equation}
for every $z\in\sn$.
\end{theorem}
Let us point out that the outermost inequalities of \eqref{eq:int_ineq_sz_bar} can be seen as a generalization of the \textit{P\'olya--Szeg\H{o} principle} (for a restricted class of functions), which is an inequality of the form
\[
\int_{\Rn} \psi(|\nabla f(x)|)\d x \geq \int_{\Rn} \psi(|\nabla f^\star (x)|) \d x
\]
for suitable \textit{Sobolev functions} $f\colon \Rn\to\R$, where $f^\star$ denotes the \textit{decreasing rearrangement} of $f$, and where $\psi\colon [0,\infty)\to[0,\infty]$ is a \textit{Young function}. Indeed, it is straightforward to reformulate Theorem~\ref{thm:int_w_nabla_u} in terms of $w^*=\phi\colon \Rn\to\R$ convex, since $w\in\fconvs$ if and only if $w^*\in\fconvf$. If we choose $\phi$ such that $\phi(o)=0$, then instead of a convex function $u\in\fconvz$ with boundary conditions, we may consider a concave function $f\colon\Rn\to\R$ with compact support (via the relation $u=-f+\ind_{\supp f}$). The outermost inequalities of \eqref{eq:int_ineq_sz_bar} now read
\[
\int_{\Rn}\phi(\nabla f(x)) \d x \geq \int_{\Rn} \phi^\bullet(\nabla f^\star (x)) \d x,
\]
where $\phi^\bullet$ is obtained from $\phi$ by taking the convex conjugate, symmetrizing the level sets, and taking the conjugate again. A more general anisotropic version of such an inequality was recently obtained by Bianchi, Cianchi, and Gronchi in \cite{Bianchi_Cianchi_Gronchi}. See also \cite{Bianchi_Gardner_Gronchi_Kiderlen_JFA24} for another recent generalization of the P\'olya--Szeg\H{o} inequality.

\paragraph{In Memoriam.} The authors dedicate this article to the memory of Paolo Gronchi, who passed away on July 4\textsuperscript{th}, 2024. He was an outstanding person and mathematician. His example reached far beyond mathematics, and his kind and lighthearted spirit has outlived all adversities.

\section{Preliminaries}
\label{se:preliminaries}

We endow $n$-dimensional Euclidean space $\Rn$ with the Euclidean norm $|\cdot|$ and the usual scalar product $\langle\cdot,\cdot\rangle$. The origin will be denoted by $o\in\Rn$, and we write $\{e_1,\ldots,e_n\}$ for the standard orthonormal basis in $\Rn$.

\medskip

It will be expedient for us to sometimes work in the $(n+1)$-dimensional ambient space $\RN$ with unit sphere $\sN$. We decompose $\sN$ into three disjoint parts: the equator,
\[
\sN_0=\{\nu \in \sN : \langle \nu,e_{n+1}\rangle = 0\},
\]
which we will sometimes identify with the $(n-1)$-dimensional unit sphere $\sn$;
the (open) lower half-sphere of dimension $n$,
\[
\sN_{-}=\{\nu\in\sN : \langle \nu, e_{n+1}\rangle <0\};
\]
and the (open) upper half-sphere $\sN_{+}$, which is defined analogously to $\sN_{-}$. We furthermore need the \textit{gnomonic projection} $\gnom\colon \sN_{-} \to\Rn$,
\[
\gnom(\nu)= \frac{(\nu_1,\dots,\nu_n)}{|\nu_{n+1}|},
\]
where $\nu=(\nu_1,\ldots,\nu_{n+1})$ for $\nu\in\sN_{-}$. Moreover, let us point out that
\[
\sqrt{1+\gnom(\nu)^2} = \frac{1}{|\nu_{n+1}|}
\]
for $\nu\in\sN_{-}$.

\subsection{Convex Bodies}
\label{se:convex_bodies}
We recall some definitions and results from the geometry of convex bodies, where we refer to the books by Gruber \cite{Gruber_book} and Schneider \cite{Schneider_CB} as standard references.

\medskip

We have already introduced the set of convex bodies (which we consider together with the Hausdorff metric), the intrinsic volumes, and the mixed volume in Section~\ref{se:geometric_ineq}. Since we will mostly work with convex bodies in $\RN$, we use the corresponding set $\KN$ in the following.

\medskip

To each body $K\in\KN$ we associate its \textit{surface area measure} $S_n(K,\cdot)$ which is a Radon measure on $\sN$ such that
\[
S_n(K,\omega)=\hm^n(\{x\in \bd K : K \text{ has an outer unit normal at } x \text{ in } \omega\}),
\]
for $\omega\subseteq \sN$. Equivalently, this measure can be described as the push-forward of $\hm^n$, restricted to $\bd K$, under the \textit{Gauss map}. The surface area measure also naturally arises from the variational formula
\begin{equation}
\label{eq:first_variation_volume}
\frac{\d}{\d t} \bigg\vert_{t=0^+}  V_{n+1}(K + t L) = \int_{\sN} h_{L}(\nu) \d S_n(K,\nu). 
\end{equation}
which holds for $K,L\in\KN$.

\medskip

Similar to \eqref{eq:mixed_vol} we have the polynomial expansion
\begin{equation}
\label{eq:mixed_area_measure}
S_n(\lambda_1 K_1+\cdots+\lambda_m K_m,\cdot) = \sum_{i_1,\ldots,i_n=1}^m \lambda_{i_1}\cdots \lambda_{i_n} S(K_{i_1},\ldots,K_{i_n},\cdot)
\end{equation}
with $m\in\N$, $K_1,\ldots,K_m\in\KN$, and $\lambda_1,\ldots,\lambda_m\geq 0$, gives rise to the \textit{mixed area measure} $S$, which associates to each $n$-tuple of convex bodies a Radon measure on $\sN$. The mixed area measure is uniquely determined by \eqref{eq:mixed_area_measure} together with the condition that $S$ is symmetric in its entries. It furthermore satisfies the relation
\begin{equation}
\label{eq:mixed_vol_int_s}
V(K_1,\ldots,K_{n+1})=\frac{1}{n+1} \int_{\sN} h_{K_1}(\nu) \d S(K_2,\ldots,K_{n+1},\nu).
\end{equation}
for $K_1,\ldots,K_{n+1}\in\KN$.

\medskip

Finally, let us mention that the intrinsic volumes are continuous with respect to the \textit{Hausdorff metric}, and similarly, the mixed volume is continuous in each of its entries. The surface area measure admits weak continuity in the sense that the integration of a continuous function results in a continuous quantity. Similarly, the mixed area measure is weakly continuous in every entry.

\subsection{Convex Functions}
\label{se:convex_functions}
In this section, we collect some results on convex functions as well as operators defined on convex functions. In addition to the books mentioned in the previous subsection and unless stated otherwise, we consider the books by Rockafellar \cite{Rockafellar} as well as Rockafellar and Wets \cite{RockafellarWets} as standard literature for the following exposition.

\medskip

For a function $w\colon \Rn\to [-\infty,\infty]$ we write
\begin{equation}
\label{eq:def_conv_conj}
w^*(y)=\sup\nolimits_{x\in\Rn} \left(\langle x,y \rangle - \varphi(x)\right),\qquad y\in\Rn,
\end{equation}
for the \textit{Legendre--Fenchel transform} or convex conjugate of $w$, which is a lower semicontinuous, convex function on $\Rn$. It holds that $u\in\fconvs$ if and only if $u^*\in\fconvf$. In addition $u^{**}=u$ for every $u\in\fconvs$ and $v^{**}=v$ for every $v\in\fconvf$.

Among the standard properties of convex conjugation are its interplay with \textit{vertical translations} in the form of
\[
(w+c)^*=w^*-c
\]
for $c\in\R$, and \textit{horizontal translations}
\[
w_o^*=w^*+\langle x_o,\cdot\rangle,
\]
where $w_o(x)=w(x-x_o)$, $x\in\Rn$, for some $x_o\in\Rn$.

\medskip

We equip all occurring spaces of convex functions with the topology associated with \textit{epi-convergence}, which admits simple descriptions on the spaces considered in this article. On $\fconvf$, epi-convergence of functions is equivalent to pointwise convergence everywhere, as well as uniform convergence on compact sets. On $\fconvs$ and its subspaces, epi-\-con\-ver\-gence of a sequence $u_j$, $j\in\N$, to $u\in\fconvs$ is equivalent to the Hausdorff convergence of the level sets $\{u_j\leq s\}\to\{u\leq s\}$ for every $s\neq \min_{x\in\Rn} u(x)$, where for $s<\min_{x\in\Rn} u(x)$ this means that $\{u_j\leq s\}=\emptyset$ for every $j\geq j_o$ with some $j_o\in\N$ (cf. \cite[Lemma 5]{Colesanti-Ludwig-Mussnig-1} and \cite[Theorem 3.1]{Beer_Rockafellar_Wets}). In particular, convex conjugation is a continuous involution between $\fconvs$ and $\fconvf$.

\medskip

For a lower-semicontinuous convex function $u$ on $\Rn$ we denote by
\[
\epi(u) = \{(x,t)\in\Rn\times\R: u(x)\leq t\}
\]
the \textit{epi-graph} of $u$, which is a convex, closed subset of $\R^{n+1}$. The addition of epi-graphs is the standard addition on $\fconvs$. That is, for $u_1,u_2\in\fconvs$ we denote by $u_1\infconv u_2\in\fconvs$ their infimal convolution which is defined as
\[
(u_1\infconv u_2)(x)=\inf\nolimits_{x_1+x_2=x} u_1(x_1)+u_2(x_2)
\]
for $x\in\R^n$. Equivalently,
\[
\epi(u_1\infconv u_2)=\epi(u_1)+\epi(u_2),
\]
where Minkowski addition is used on the right side. Furthermore, we define epi-multiplication on $\fconvs$ as the corresponding multiplication with a scalar. That is, for $\lambda>0$ and $u\in\fconvs$, we set
\[
\lambda\sq u(x)=\lambda\, u\left( \frac x\lambda \right)
\]
for $x\in\R^n$ and remark that this continuously extends to $0\sq u=\ind_{\{o\}}$. Observe that
\[
\left(\lambda \sq \ind_K\right) \infconv \left(\mu \sq \ind_L\right) = \ind_{\lambda K + \mu L}
\]
for every $K,L\in\Kn$ and $\lambda,\mu\geq 0$.

The composition of epi-convergence with convex conjugation leads to pointwise addition, that is
\[
\big((\lambda_1 \sq u_1)\infconv(\lambda_2\sq  u_2)\big)^*=\lambda_1 u_1^* + \lambda_2 u_2^*
\]
for $u_1,u_2\in\fconvs$ and $\lambda_1,\lambda_2\geq 0$. In this context, let us point out that the $\ind_K^*=h_K$ for every $K\in\Kn$.

\medskip

To a convex function $v\in\fconvf$ we associate the Monge--Amp\`ere measure $\MA(v;\cdot)$, which is a Radon measure on $\Rn$ and which is defined by
\[
\MA(v;B)=V_n\left(\bigcup_{x\in B} \partial v(x) \right)
\]
for every Borel set $B\subset \Rn$. Here,
\[
\partial v(x)=\{y\in\Rn : v(z)\geq v(x)+\langle y,z-x\rangle \text{ for all } z\in\Rn\}
\]
denotes the \textit{subdifferential} of $v$ at $x\in\Rn$. Under additional $C^2$ assumptions on $v$, this measure is absolutely continuous with respect to the Lebesgue measure on $\Rn$ such that the representation \eqref{eq:MA_Hess} holds. Just like the mixed area measure arises from a polynomial expansion of the surface area measure, we have
\[
\MA(\lambda_1 v_1+\cdots+\lambda_m v_m;\cdot)=\sum_{i_1,\ldots,i_n=1}^m \lambda_{i_1}\cdots\lambda_{i_n} \MA(v_{i_1},\ldots,v_{i_n};\cdot)
\]
for $m\in\N$, $v_1,\ldots,v_m\in\fconvf$, and $\lambda_1,\ldots,\lambda_m\geq 0$, where the coefficients on the right are the so-called \textit{mixed Monge--Amp\`ere measures}. These are again defined to be symmetric in their entries and assign to each $n$-tuple of functions in $\fconvf$ a Radon measure on $\Rn$. If $v_1,\ldots,v_n\in\fconvf\cap C^2(\Rn)$, then
\[
\d\MA(v_1,\ldots,v_n;x)=D(\Hess v_1(x),\ldots,\Hess v_n(x)) \d x.
\]
Let us point out that the Monge--Amp\`ere measure and, more generally, the mixed Monge--Amp\`ere measure are invariant under the addition of affine functions, that is,
\[
\MA(v_1+a,\ldots,v_n;\cdot)=\MA(v_1,\ldots,v_n;\cdot)
\]
for every $v_1,\ldots,v_n\in\fconvf$ and affine $a\colon \Rn\to\Rn$. In particular, the measure is invariant under vertical translations of the functions. For additional information, see, for example, \cite[Theorem 4.3]{Colesanti-Ludwig-Mussnig-7}.

\medskip

Considering the Monge--Amp\`ere measure of the convex conjugate of a function leads to the \textit{conjugate Monge--Amp\`ere measure}. That is, to each $u\in\fconvs$ we assign the measure $\MAp(u;\cdot)=\MA(u^*;\cdot)$. This measure admits the representation
\begin{equation}
\label{eq:MAp}
\MAp(u;B)=\int_{\dom(u)}\chi_B(\nabla u(x)) \d x,
\end{equation}
for every Borel set $B\subset \Rn$, where $\chi_B$ denotes the \textit{characteristic function} of $B$. This means that $\MAp$ is the push-forward of the Lebesgue measure, restricted to $\dom(u)$, under the gradient of $u$. Similarly, we define the conjugate mixed Monge--Amp\`ere measure via
\[
\MAp(u_1,\ldots,u_n;\cdot)=\MA(u_1^*,\ldots,u_n^*;\cdot)
\]
for $u_1,\ldots,u_n\in\fconvs$. Observe that also the \textit{conjugate mixed Monge--Amp\`ere measure} is invariant under vertical translations of its entries.

\subsection{Functional Intrinsic Volumes}
\label{se:functional_intrinsic_volumes}
For $0\leq j\leq n$ and $\alpha\in C_c([0,\infty))$ we have introduced the (renormalized) $j$ functional intrinsic volume with density $\alpha$ in \eqref{eq:ozzb_steiner}. As was shown in \cite[Theorem 5.2]{Colesanti-Ludwig-Mussnig-7}, this operator can also be described via
\begin{equation}
\label{eq:ozzb_ma}
\oZZb{j}{\alpha}(u)=\binom{n}{j}\frac{1}{\kappa_{n-j}}\int_{\Rn} \alpha(|x|)\d \MA^*(u[j],\ind_{B^n}[n-j];x)
\end{equation}
for $u\in\fconvs$, where $u[j]$ indicates that the entry $u$ in the conjugate mixed Monge--Amp\`ere measure is repeated $j$ times. The properties of conjugate mixed Monge--Amp\`ere measures together with \eqref{eq:ozzb_ma} immediately show that
\[
\oZZb{j}{\alpha}(\lambda \sq u) = \lambda^j \, \oZZb{j}{\alpha}(u)
\]
for every $\lambda\geq 0$ and $u\in\fconvs$ and we say that $\oZZb{j}{\alpha}$ is \textit{epi-homogeneous} of degree $j$.

\medskip

We require the following result due to \cite[Lemma 8.4]{Colesanti-Ludwig-Mussnig-7}, which shows how to retrieve the densities $\alpha$ from $\oZZb{j}{\alpha}$ (cf.\ \cite[Lemma 2.15]{Colesanti-Ludwig-Mussnig-5}). For this, we will consider the family of functions $u_t\in\fconvs$, $t\geq 0$, given by
\begin{equation}
\label{eq:def_u_t}
u_t(x)= t \vert x\vert + \ind_{\Bn}(x)
\end{equation}
for $x\in\Rn$. Observe that $u_t$ has compact domain and is, therefore, also an element of the smaller space $\fconvcd$.

\begin{lemma}
\label{le:ozz_on_u_t}
If $\,1\leq j \leq n$ and  $\alpha \in C_c([0,\infty))$, then  
\[
\oZZb{j}{\alpha}(u_t)=\frac{\kappa_n}{\kappa_{n-j}}\binom{n}{j} \alpha(t)
\]
for $t\geq 0$.
\end{lemma} 
\noindent
As a consequence, we obtain the following characterization of non-negativity, which was first obtained in \cite[Proposition 5.4]{Colesanti-Ludwig-Mussnig-6}.
\begin{lemma}
\label{le:ozz_sign}
For $1\leq j\leq n$ and $\alpha \in C_c([0,\infty))$, the operator $\oZZb{j}{\alpha}$ is non-negative on $\fconvs$ if and only if $\alpha$ is non-negative.
\end{lemma}

\medskip

Similar to the classical intrinsic volumes on convex bodies, there are integral-geometric identities for functional intrinsic volumes. The following \textit{Cauchy--Kubota-type formulas} were first established for rotation invariant integrands in \cite{Colesanti-Ludwig-Mussnig-5} and later, in a more general form, in \cite[Theorem 5.1]{Hug_Mussnig_Ulivelli}. For $u\in\fconvs$ and a linear subspace $E\subseteq \Rn$ we write $\proj_E u\colon E\to (-\infty,\infty]$ for the \textit{projection function} of $u$, defined by
$$\proj_E u(x_E) = \min\nolimits_{z\in E^\perp} u(x_E+z)$$
for $x_E\in E$, where $E^\perp$ denotes the orthogonal complement of $E$. In addition, we denote by $\MAp_E$ the conjugate mixed Monge--Amp\`ere measure of super-coercive convex functions defined on $E$.

\begin{theorem}
\label{thm:ck_ma}
If $1\leq k< n$ and $\varphi\in C_c(\Rn)$, then
\begin{align*}
\frac{1}{\kappa_n} \int_{\Rn} \varphi(x)&\d\MA^*(u_1,\ldots,u_k,\ind_{B^n}[n-k];x)\\
&=\frac{1}{\kappa_k}\int_{\Grass{k}{n}}\int_E \varphi(x_E) \d\MA^*_E(\proj_E u_1,\ldots, \proj_E u_k;x_E)\d E
\end{align*}
for every $u_1,\ldots,u_k\in\fconvs$, where we integrate with respect to the Haar probability measure on the Grassmannian $\Grass{k}{n}$.
\end{theorem}

\section{Counterexamples}
\label{se:counterexample}
\subsection{Brunn--Minkowski-type Inequalities}
\label{se:bm_ineq}
The general Brunn--Minkowski inequality for convex bodies, \eqref{eq:bm_general}, together with \eqref{eq:ozz_on_ind} and the epi-homogeneity of functional intrinsic volumes show that
\begin{equation}
\label{eq:bm_ozz_ind}
\left(\oZZb{j}{\alpha}\big((1-\lambda)\sq\ind_K \infconv \lambda\sq\ind_L\big)\right)^{\frac 1j} \geq (1-\lambda) \left(\oZZb{j}{\alpha}(\ind_K) \right)^{\frac 1j} + \lambda \left(\oZZb{j}{\alpha}(\ind_L) \right)^{\frac 1j}
\end{equation}
for every $1\leq j\leq n$, $K,L\in \Kn$, $0\leq\alpha\leq 1$, and $\alpha\in C_c([0,\infty))$ such that $\alpha(0)\geq 0$. Given the similarities between functional intrinsic volumes and their classical counterparts, it is therefore only natural to ask whether such inequalities also hold if we replace indicator functions in \eqref{eq:bm_ozz_ind} with more general functions $u,v\in\fconvs$. That is, for $1\leq j\leq n$ and suitable $\alpha\in C_c([0,\infty))$, we ask if
\begin{equation}
\label{eq:bm_ozzb_question}
\left(\oZZb{j}{\alpha}\big((1-\lambda)\sq u \infconv \lambda\sq v\big)\right)^{\frac 1j} \geq (1-\lambda) \left(\oZZb{j}{\alpha}(u) \right)^{\frac 1j} + \lambda \left(\oZZb{j}{\alpha}(v) \right)^{\frac 1j}
\end{equation}
holds for every $u,v\in\fconvs$ and $0\leq \lambda \leq 1$.

Indeed, in case $j=1$, inequality \eqref{eq:bm_ozzb_question} is always true and is, in fact, an equality since the operators $\oZZb{1}{\alpha}$ are linear with respect to infimal convolution. For $j\geq 2$, however, we show that the Brunn--Minkowski-type inequality \eqref{eq:bm_ozzb_question} does not hold in general, even when restricting to convex functions with compact domains. Note that by Lemma~\ref{le:ozz_sign}, such an inequality would only be meaningful if $\alpha$ is non-negative. Let us furthermore recall the definition of $u_t$ in \eqref{eq:def_u_t}.

\begin{proposition}
\label{prop:BM_conv}
For every $2\leq j\leq n$ and non-negative $\alpha \in C_c([0,\infty))$, $\alpha \not\equiv 0$, there exist $t_1,t_2\geq 0$ such that
\[
\left(\oZZb{j}{\alpha}(u_{t_1} \infconv u_{t_2})\right)^{\frac 1j} <  \left(\oZZb{j}{\alpha}(u_{t_1}) \right)^{\frac 1j} + \left(\oZZb{j}{\alpha}(u_{t_2}) \right)^{\frac 1j}.
\]
\end{proposition}
\begin{proof}
Let $2\leq j\leq n$ and let a non-negative $\alpha\in C_c([0,\infty))$, $\alpha\not\equiv 0$ be given. By the properties of $\alpha$, there exist $t_1< t_2$ such that
\begin{equation}
\label{eq:alpha_ineq}
\alpha(t_1)>\alpha(t_2).
\end{equation}
For the infimal convolution of $u_{t_1}$ and $u_{t_2}$ we now have
\[
(u_{t_1}\infconv u_{t_2}) (x) = \begin{cases}
t_1 |x|,\quad &\text{for } 0\leq |x| \leq 1\\
t_2 (|x|-1)+t_1,\quad &\text{for } 1 < |x| \leq 2\\
+\infty, \quad &\text{for } 2 < |x|.
\end{cases}
\]
Thus, it follows from Theorem~\ref{thm:ck_ma} that
\begin{align*}
\oZZb{j}{\alpha}(u_{t_1}\infconv u_{t_2}) &= 
\binom{n}{j}\frac{\kappa_n}{\kappa_j\kappa_{n-j}}\int_{\Grass{j}{n}}\int_E \alpha(|x_E|) \d\MA^*_E(\proj_E (u_{t_1} \infconv u_{t_2});x_E)\, dE\\
&=\binom{n}{j}\frac{\kappa_n}{\kappa_j\kappa_{n-j}}\int_{\Grass{j}{n}}\int_{\dom (\proj_E (u_{t_1}\infconv u_{t_2}))} \alpha(|\nabla \proj_E (u_{t_1}\infconv u_{t_2})(x_E)|) \d x_E\, dE\\
&= \frac{\kappa_n}{\kappa_j\kappa_{n-j}} \binom{n}{j} \left(\kappa_j \alpha(t_1) + (2^j-1)\kappa_j \alpha(t_2) \right)\\
&= \frac{\kappa_n}{\kappa_{n-j}} \binom{n}{j} \left(\alpha(t_1)+(2^j-1)\alpha(t_2)\right).
\end{align*}
Together with \eqref{eq:alpha_ineq} and Lemma~\ref{le:ozz_on_u_t} this shows
\begin{align*}
\oZZb{j}{\alpha}(u_{t_1}\infconv u_{t_2}) &= \frac{\kappa_n}{\kappa_{n-j}} \binom{n}{j} \left(\alpha(t_1) + (2^j-1) \alpha(t_2)\right)\\
&= \frac{\kappa_n}{\kappa_{n-j}} \binom{n}{j}\left(\binom{j}{j} \alpha(t_1) + \alpha(t_2) \sum_{i=0}^{j-1} \binom{j}{i}\right)\\
&< \frac{\kappa_n}{\kappa_{n-j}} \binom{n}{j} \sum_{i=0}^j \binom{j}{i} \alpha(t_1)^{\frac ij} \alpha(t_2)^{\frac{j-i}{j}}\\
&=\left(\left(\oZZb{j}{\alpha}(u_{t_1}) \right)^{\frac 1j} + \left(\oZZb{j}{\alpha}(u_{t_2}) \right)^{\frac 1j}\right)^j,
\end{align*}
which completes the proof.
\end{proof}

In the following, we give a more in-depth explanation of why inequality \eqref{eq:bm_ozzb_question} fails in general. For simplicity, we only consider the case $j=n$ below. We start by recalling a result which is essentially due \cite[Corollary 2]{Knoerr_Ulivelli} and the formulation below is an immediate consequence of \cite[Lemma 4.10]{Hug_Mussnig_Ulivelli} and \cite[Lemma 6.1]{Hug_Mussnig_Ulivelli_2}. We furthermore remark that we will discuss this connection between convex functions on $\Rn$ and convex bodies in $\RN$ in more detail in Section~\ref{suse:bodies_fcts}.

\begin{lemma}
\label{le:ozzb_int_sn}
For every $\alpha\in C_c([0,\infty))$ and $u\in\fconvs$ there exists a convex body $K\in\KN$ such that
\begin{equation}
\label{eq:ozzb_int_sn}
\oZZb{n}{\alpha}(u) = \int_{\sN_-} \tilde{\alpha}(|\langle \nu,e_{n+1} \rangle|) \d S_n(K,\nu),
\end{equation}
where
\begin{equation}
\label{eq:tilde_alpha}
\tilde{\alpha}(|\langle \nu,e_{n+1} \rangle|) = \frac{\alpha(|\gnom(\nu)|)}{\sqrt{1+|\gnom(\nu)|^2}}
\end{equation}
for $\nu\in\sN_-$. Conversely, for every $K\in\KN$ there exists $u\in\fconvcd$ such that \eqref{eq:ozzb_int_sn} holds for every $\alpha\in C_c([0,\infty))$.
\end{lemma}

The geometric phenomenon that explains the failure of a Brunn--Minkowski-type inequality for functional intrinsic volumes is due to \cite[Theorem 1.1]{Colesanti-Hug-Saorin_Gomez_JGA}. We formulate this result with respect to the ambient space $\R^{n+1}$. See also \cite{Colesanti-Hug-Saorin_Gomez_CCM}.

\begin{theorem}
\label{thm:BM_sn}
Let $f\in C(\sN)$ and set
\[
F(K)=\int_{\sN} f(\nu) \d S_n(K,\nu)
\]
for $K\in\KN$. The inequality
\begin{equation}
\label{eq:BM_sn}
F(K+L)^{\frac 1n} \geq F(K)^{\frac 1n} + F(L)^{\frac 1n}
\end{equation}
holds for every $K,L\in\KN$, if and only if $f$ is the restriction of a support function to $\sN$.
\end{theorem}

\begin{proof}[Alternative proof of Proposition~\ref{prop:BM_conv} for $j=n$]
Let $\alpha\in C_c([0,\infty))$ be given and let $\tilde{\alpha}$ be as in \eqref{eq:tilde_alpha}. We define $f\in C(\sN)$ as
\[
f(\nu) = \begin{cases}
\tilde{\alpha}(|\langle \nu,e_{n+1}\rangle|)\quad &\text{if } \nu\in \sN_{-},\\
0\quad &\text{else}.
\end{cases}
\]
By Lemma~\ref{le:ozzb_int_sn} the inequality
\[
\left(\oZZb{n}{\alpha}(u\infconv v)\right)^{\frac 1n}\geq \left(\oZZb{n}{\alpha}(u)\right)^{\frac 1n}+\left(\oZZb{n}{\alpha}(v)\right)^{\frac 1n}
\]
holds for every $u,v\in\fconvs$ if and only if \eqref{eq:BM_sn} holds for every $K,L\in\KN$. Here, we have used that if $K$ and $L$ are the convex bodies associated with $u$ and $v$, respectively, then $K+L$ is the body associated with $u\infconv v$ (we postpone the proof of this statement to Lemma~\ref{le:infconv_ku} below). By Theorem~\ref{thm:BM_sn}, this implies that $f$ must be the restriction of a support function to $\sN$. However, since $\alpha$ has compact support, $f$ must vanish in a neighborhood of the equator $\{\nu\in \sN : \langle \nu,e_{n+1}\rangle =0\}$. The only support function that satisfies this property is $h_{\{o\}}$, which corresponds to the trivial case $\alpha\equiv 0$.
\end{proof}

\subsection{Isoperimetric Inequalities}
Throughout the following, let $1\leq j <k \leq n$. The isoperimetric inequality and its more general form \eqref{eq:iso} can be stated as follows: among all convex bodies with fixed $k$th intrinsic volume, Euclidean balls minimize the $j$th intrinsic volume.

As in the case of Brunn--Minkowski-type inequalities, it is a consequence of Lemma~\ref{le:ozz_sign} that analogs of such inequalities for functional intrinsic volumes can only be expected for non-negative densities $\alpha \in C_c([0,\infty))$. By \eqref{eq:ozz_on_ind}, one retrieves intrinsic volumes of convex bodies from convex indicator functions, and therefore, the same density $\alpha$ should appear on both sides of an isoperimetric-type inequality for functional intrinsic volumes. Thus, it is reasonable to expect an inequality (or isoperimetric principle) that compares the functionals 
\[\oZZb{j}{\alpha} \text{ and }\oZZb{k}{\alpha}\]
for some non-negative $\alpha\in C_c([0,\infty))$. In order to show that such an isoperimetric principle fails, we need the following consequence of \cite[Corollary 1.5]{Hug_Mussnig_Ulivelli} (see also \cite[Theorem 5.11]{Hug_Mussnig_Ulivelli}).

\begin{lemma}
\label{le:ozzb_vanish}
Let $1\leq j\leq k\leq n$ and let $\alpha\in C_c({[0,\infty)})$ be non-negative. If $u\in\fconvs$ is such that $\oZZb{j}{\alpha}(u)=0$, then also $\oZZb{k}{\alpha}(u)=0$.
\end{lemma}

We can now show that an isoperimetric principle fails for functional intrinsic volumes on $\fconvs$.

\begin{proposition}
\label{prop:iso_counter}
Let $1\leq j<k\leq n$, let $\alpha\in C_c([0,\infty))$, $\alpha\not\equiv 0$, be non-negative. The operator $\oZZb{j}{\alpha}$ does not attain a minimum on
\[
M_{k,\alpha}=\left\{u\in\fconvs \colon \oZZb{k}{\alpha}(u)=1 \right\}.
\]
In fact, $\oZZb{j}{\alpha}(u)>0$ for every $u\in M_{k,\alpha}$ while
\[
\inf \left\{\oZZb{j}{\alpha}(u) : u\in M_{k,\alpha}\right\}=0.
\]
\end{proposition}
\begin{proof}
For every $c\geq 0$ and every $t\geq 0$ such that $\alpha(t)\neq 0$, it follows from Lemma~\ref{le:ozz_on_u_t} that
\[
\oZZb{k}{\alpha}\left(\sqrt[k]{\frac{c}{\alpha(t)}} \sq u_t\right) = \frac{c}{\alpha(t)} \oZZb{k}{\alpha}(u_t) = \frac{c}{\alpha(t)}\frac{\kappa_n}{\kappa_{n-k}} \binom{n}{k} \alpha(t).
\]
Thus, for $c_{n,k}=\left( \frac{\kappa_n}{\kappa_{n-k}} \binom{n}{k}\right)^{-1}$ we have
\[
\sqrt[k]{\frac{c_{n,k}}{\alpha(t)}} \sq u_t \in M_{k,\alpha}
\]
for every $t\geq 0$ such that $\alpha(t)\neq 0$. Using Lemma~\ref{le:ozz_on_u_t} again, it follows that
\begin{align*}
\oZZb{j}{\alpha}\left(\sqrt[k]{\frac{c_{n,k}}{\alpha(t)}} \sq u_t \right) &= \left(\frac{c_{n,k}}{\alpha(t)}\right)^{\frac jk} \oZZb{j}{\alpha}(u_t)\\
&=\left(\frac{c_{n,k}}{\alpha(t)} \right)^{\frac jk} \frac{\kappa_n}{\kappa_{n-j}} \binom{n}{j} \alpha(t)\\
&=\frac{c_{n,k}^{\frac jk}\kappa_n}{\kappa_{n-j}} \binom{n}{j} \alpha(t)^{\frac{k-j}{k}}.
\end{align*}
Since $\alpha\in C_c([0,\infty))$ there exists a sequence of non-negative numbers $t_l$, $l\in\N$, such that $\alpha(t_l)\neq 0$ for every $l\in\N$ and $\alpha(t_l)\to 0$ as $l\to \infty$. Together with $k-j>0$ it follows that
$$\oZZb{j}{\alpha}\left(\sqrt[k]{\frac{c_{n,k}}{\alpha(t_l)}} \sq u_t \right) \to 0$$
as $l\to\infty$.

It remains to show that $\oZZb{j}{\alpha}(u)>0$ for every $u\in M_{k,\alpha}$, which is an immediate consequence of the definition of $M_{k,\alpha}$ and Lemma~\ref{le:ozzb_vanish}.
\end{proof}
\begin{remark}
The proof of Proposition~\ref{prop:iso_counter} only uses functions of the type $\lambda \sq u_t$ with $\lambda,t\geq 0$. Thus, we may replace $\fconvs$ in the definition of $M_{k,\alpha}$ with the smaller space $\fconvcd$, which means that an isoperimetric principle for functional intrinsic volumes even fails for convex functions with compact domains.\dssymb
\end{remark}

\section{Wulff-type Inequalities}
\label{se:wulff_type}
\subsection{Wulff Shapes}
Throughout the following exposition, we consider the geometric setting with respect to $(n+1)$-dimensional Euclidean space. For a continuous function $\eta\colon \sN \to (0,\infty)$, the \textit{Wulff shape} associated with $\eta$ is the set
\begin{align}
\begin{split}
\label{eq:def_wulff}
W_\eta&=\{\xi \in \RN: \langle \xi, \nu \rangle \leq \eta(\nu) \text{ for every } \nu \in \sN\}\\
&=\bigcap_{\nu\in\sN} H_{\nu}^-(\eta(\nu)),
\end{split}
\end{align}
where we write $H^-_{\nu}(t)$ for the closed half-space orthogonal to $\nu$ at distance $t$ from the origin such that $H^-_{\nu}(t)$ is unbounded in direction $-\nu$. The naming of this shape is due to the appearance of such a construction in the work by Wulff \cite{wulff_01}. Since $W_\eta$ is the intersection of closed half-spaces, it is easy to see that $W_\eta\in\KN$ and furthermore, since $\eta$ is positive and continuous, that $W_\eta$ contains the origin in its interior. We remark that this body is sometimes also called \textit{Aleksandrov body}.

Wulff's theorem, which is essentially a consequence of the Brunn--Minkowski inequality, states that the Wulff shape associated with $\eta$ is the shape that minimizes the surface integral of $\eta$ (composed with the Gauss map) among all shapes with the same volume. We refer to \cite{Figalli_Zhang} for more recent developments surrounding this result, as well as \cite{schuster_weberndorfer} for further inequalities regarding Wulff shapes. For the following formulation of this isoperimetric-type inequality, we refer to \cite{Busemann_Wulff}. See also, for example, \cite[Theorem 4.2]{Fon} and \cite[Theorem 8.13]{Gruber_book}. 

\begin{theorem}[Wullf's theorem]
\label{thm:wulff}
If $\eta\colon  \sN \to (0,+\infty)$ is continuous, then
\begin{equation}
\label{eq:wulff}
\int_{\sN} \eta(\nu) \d S_n(K) \geq (n+1) V_{n+1}^{n/(n+1)}(K) V_{n+1}^{1/(n+1)}(W_\eta)
\end{equation}
for every $K\in\KN$. If $V_{n+1}(K)>0$, then equality holds if and only if $K$ is homothetic to $W_\eta$.
\end{theorem}

Note that the Wulff shape associated with $\eta$ satisfies
\begin{equation}
\label{eq:volume_wulff_shape}
V_{n+1}(W_\eta) = \frac{1}{n+1} \int_{\sN} \eta(\nu) \d S_n(W_\eta).
\end{equation}
See, for example, \cite[Lemma 7.5.1]{Schneider_CB}. Thus, considering that $V_{n+1}(W_\eta)>0$, we can rewrite \eqref{eq:wulff} as
\begin{equation}
\label{eq:wulff_alt}
\left( \frac{\int_{\sN} \eta(\nu) \d S_n(K)}{\int_{\sN} \eta(\nu) \d S_n(W_\eta)}  \right)^{n+1} \geq  \left(\frac{V_{n+1}(K)}{V_{n+1}(W_\eta)}\right)^{n} .
\end{equation}
Let us further remark that Wulff's theorem is a generalization of \textit{Minkowski's first inequality} and that the latter is retrieved when we choose $\eta$ to be the support function of a convex body. In particular, if $\eta\equiv 1$, then we retrieve the Euclidean isoperimetric inequality.

\begin{remark}
Wulff shapes were originally considered in a discrete setting, and, in particular, it is straightforward to generalize the construction above to include the polytopal case as well. To do so, one fixes a set $\Omega\subseteq \sN$ that positively spans $\RN$. Wulff shapes are then defined as in \eqref{eq:def_wulff} with the difference that $\sN$ needs to be replaced by $\Omega$ and the remainder of the exposition above needs to be changed accordingly (see, for example, \cite[Section 7.5]{Schneider_CB}). In particular, such a generalization is also possible in the functional setting we describe in the following Section~\ref{suse:bodies_fcts}. However, to avoid overly cumbersome expressions and to keep statements simple, we have refrained from using this approach.\dssymb
\end{remark}

\subsection{From Convex Bodies to Convex Functions}
\label{suse:bodies_fcts}
We want to use Theorem~\ref{thm:wulff} to establish an inequality for convex functions. As pointed out in the alternative proof of Proposition~\ref{prop:BM_conv}, integrating a function with compact support in $\Rn$ against the conjugate Monge--Amp\`ere measure of a function in $\fconvs$ can be rewritten as integrating a function that vanishes in a neighborhood of the equator $\sN_0$ against the surface area measure of a convex body in $\KN$. Thus, in order to obtain non-trivial inequalities for convex functions, we will look at integral operators whose integrands do not have compact domains. For this purpose, we need to consider a smaller set of admissible convex functions and thus consider the space
\begin{align*}
\fconvz=\{ u\colon \Rn\to (-\infty,\infty]\,:\;& u \text{ is convex, proper, } \dom(u) \text{ is compact,}\\
&\max\nolimits_{x \in \dom(u)}u(x)\leq 0 \},
\end{align*}
which is a subset of $\fconvcd$. Let us emphasize that we could have already restricted to $\fconvz$ for the counterexamples we presented in Section~\ref{se:counterexample}. Indeed, functional intrinsic volumes are invariant under vertical translations, and for every $t\geq 0$, we have $u_t-t \in\fconvz$.

The following exposition is a consequence of \cite[Section 3]{Knoerr_Ulivelli}, where a variation of this construction is considered. Let $H\subset \RN$ denote $n$-dimensional linear space spanned by $\{e_1,\ldots,e_n\}$, i.e.\ the first $n$ vectors of the standard orthonormal basis of $\RN$. Equivalently $H=e_{n+1}^\perp$. By $\refl_H\colon \RN\to\RN$ we denote the \textit{orthogonal reflection} at $H$, that is
\[
\refl_H(x,t)=(x,-t)
\]
for $(x,t)\in\Rn\times \R$. To each function $u\in\fconvz$ we associate the convex body
\begin{equation}\label{eq:ref_Ku}
    K^u= \epi(u)\cap \refl_H(\epi(u)),
\end{equation}
where we write $\refl_H(C)=\bigcup_{c\in C} \refl_H(c)$ for $C\subset \RN$. From the properties of $u$ it is clear that $K^u \in \KN$ and that $\refl_H(K^u) = K^u$. Motivated by this, we introduce the notation
$$\KNH = \{K\in\KN : \refl_H(K) = K\}$$
and see that we have thus defined a map
\begin{align}
\begin{split}
\label{eq:ku}
\fconvz &\to \KNH\\
u &\mapsto K^u.
\end{split}
\end{align}
Conversely, to each $K \in \KNH$ we can associate a function
\begin{align}
\label{eq:boundary_map}
\lfloor K \rfloor (x):=
\begin{cases}
\inf_{(x,t)\in K}t \quad &\text{if  }x\in\proj_H K,\\
\infty \quad &\text{else},
\end{cases}
\end{align}
where $\proj_H\colon \RN\to H$ denotes the \textit{orthogonal projection} onto $H$. By definition, the epi-graph of $\lfloor K \rfloor$ is convex. Furthermore, the compactness of $K$ implies that $\dom(\lfloor K \rfloor)$ is compact. Thus, we have $\lfloor K \rfloor \in \fconvcd$. In addition, it follows from the definition of $\KNH$ that
\[
\max\nolimits_{x\in\dom(\lfloor K \rfloor)} \lfloor K \rfloor (x) \leq 0,
\]
and therefore $\lfloor K \rfloor\in\fconvz$. In particular, the map $K\mapsto \lfloor K \rfloor$ is the inverse of $u\mapsto K^u$ and it is straightforward to see that
\begin{equation}
\label{eq:ku_floor}
\lfloor K^u\rfloor = u
\end{equation}
for every $u\in\fconvz$ and conversely
\begin{equation}
\label{eq:kfloork}
K^{\lfloor K \rfloor} = K
\end{equation}
for every $K\in\KNH$. For the functional $\oVb_{n+1}$, defined in \eqref{eq:def_ovb_n+1}, we immediately obtain 
\begin{equation}
\label{eq:oVb_V}
\oVb_{n+1}(u) = \int_{\dom(u)}|u(x)|\d x = \frac 12 V_{n+1}(K^u)
\end{equation}
for every $u\in\fconvz$. Let us also point out that for the convex conjugate of 
$\lfloor K\rfloor$ we have
\begin{equation}
\label{eq:kfloor_conj_hk}
\lfloor K\rfloor^*(x)= h_K(x,-1)
\end{equation}
for $x\in\Rn$.

Lastly, the following consequence of the area formula for Lipschitz functions, which is due to \cite[Corollary 2]{Knoerr_Ulivelli}, connects integrals with respect to $\MAp(u;\cdot)$ (cf.\ \eqref{eq:MAp}) and with respect to $S_n(K^u,\cdot)$. Note that the setting in \cite{Knoerr_Ulivelli} is slightly different, but it is straightforward to check that the result also applies to the setting that is considered here.

\begin{lemma}
\label{le:int_ku}
If $\varphi\colon \Rn \to \R$ is bounded and Borel measurable, then
\[
\int_{\sN_{-}} \tfrac{\varphi(\gnom(\nu))}{\sqrt{1+|\gnom(\nu)|^2}} \d S_n(K^u,\nu) = \int_{\dom(u)} \varphi(\nabla u(x)) \d x
\]
for every $u\in\fconvz$.
\end{lemma}

We will now extend Lemma~\ref{le:int_ku} to functions defined on the whole sphere $\sN$ or equivalently to include boundary terms in the analytic setting. For this purpose, recall that the recession function $\rho_{\varphi}$ of $\varphi\in C(\Rn)$ is the $1$-homogenous extension of $z\mapsto \lim_{t\to\infty} \varphi(tz)/t$ (see \eqref{eq:rho_varphi}) and the $\Cr$ denotes the set of all $\varphi\in C(\Rn)$ that have a continuous recession function such that $|\rho_\varphi-\varphi|$ is bounded on $\Rn$ (see \eqref{eq:def_Cr}). For $\varphi\in\Cr$ we now consider the functional $\oZ_\varphi\colon \fconvz\to\R$, given by
\begin{equation}
\label{eq:oz_varphi}
\oZ_\varphi(u)=\int_{\dom(u)} \varphi(\nabla u(x))\d x + \int_{\bd(\dom(u))} |u(x)| \rho_{\varphi}(\nu_{\dom(u)}(x))\d\hm^{n-1}(x)
\end{equation}
for $u\in\fconvz$ such that $\dim(\dom(u))=n$. Here, $\bd(\dom(u))$ denotes the boundary of $\dom(u)$ and $\nu_{\dom(u)}(x)\in\sn$ is the outer unit normal to $\dom(u)$ at $x\in \bd(\dom(u))$, which is well-defined $\hm^{n-1}$-a.e.\ on $\bd(\dom(u))$. Note that our conditions on $\varphi$, together with the fact that elements of $\fconvz$ have compact domain, guarantee that \eqref{eq:oz_varphi} is finite. Furthermore, observe that the second integral on the right side of \eqref{eq:oz_varphi} is, in fact, an integral on $\sn$ with respect to the measure $\Sa(u;\cdot)$, which we define as follows: if $\dim(\dom(u))=n$, then $\Sa(u;\cdot)$ is the push-forward of $\nu_{\dom(u)}$ under $|u(x)|\d\hm^{n-1}(x)$ restricted to $\bd(\dom(u))$. That is
\begin{equation}
\label{eq:s_pushforward}
\d \Sa(u;\cdot)=(\nu_{\dom(u)})_{\sharp}\left(|u| \d \hm^{n-1}\Big\vert_{\bd(\dom(u))} \right).
\end{equation}
If $\dim(\dom(u))<n$, then $\dom(u)$ is contained in a hyperplane with outer unit normal $z_u\in \sn$. In this case, the measure $\Sa(u;\cdot)$ is a discrete measure concentrated on $\pm z_u$ such that
\begin{equation}
\label{eq:sa_lower_dim}
\Sa(u;\{z_u\})=\Sa(u;\{-z_u\}) = \int_{\bd(\dom(u))}|u(x)| \d\hm^{n-1}(x).
\end{equation}
Another way of looking at the measure we just defined is the fact that 
\begin{equation}
\label{eq:sa_sn}
\Sa(u;\omega)=\frac 12 S_n(K^u,\omega)\Big\vert_{\sN_0},
\end{equation}
for every $u\in\fconvz$ and Borel set $\omega \subset \sN_0$. That is, $\Sa(u;\cdot)$ is the restriction of the usual surface area measure of $K^u$ to the equator $\sN_0$, which we identify with $\sn$. Lastly, let us remark that similar measures appear, for example, in \cite{Colesanti_Fragala,Rotem_JFA_22,Rotem_23,Ulivelli_Wulff}.

To summarize, we rewrite \eqref{eq:oz_varphi} in the form
\begin{equation}
\label{eq:oz_varphi_measures}
\oZ_\varphi(u)=\int_{\Rn} \varphi(y)\d\MAp(u;y) + \int_{\sn} \rho_{\varphi}(z) \d \Sa(u;z).
\end{equation}
As noted in the introduction, we want to emphasize that $\oZ_\varphi$ generalizes the functional intrinsic volume $\oZZb{n}{\alpha}$. Indeed, if $\varphi(x)=\alpha(|x|)$ with some $\alpha\in C_c([0,\infty))$, then $\rho_{\varphi}\equiv 0$ and we have
\[
\oZ_\varphi(u)= \int_{\Rn} \alpha(|y|) \d\MAp(u;y)=\oZZb{n}{\alpha}(u)
\]
for $u\in\fconvz$.

Considering that for $u\in\fconvz$ the associated body $K^u$ is symmetric with respect to the hyperplane $H$, we immediately obtain the following consequence of Lemma~\ref{le:int_ku}, \eqref{eq:sa_sn} and \eqref{eq:oz_varphi_measures}.

\begin{lemma}
\label{le:varphi_tilde}
Let $\varphi\in \Cr$ and set
\begin{equation}
\label{eq:varphi_tilde}
\tilde{\varphi}(\nu)=
\begin{cases}
    \frac{\varphi(\gnom(\nu))}{\sqrt{1+|\gnom(\nu)|^2}}\quad & \text{if } \nu_{n+1} <0,\\[1.4ex]
    \rho_\varphi((\nu_1,\ldots,\nu_n)) \quad &\text{if } \nu_{n+1} =0,\\[1.4ex]
    \frac{\varphi(\gnom(-\nu))}{\sqrt{1+|\gnom(-\nu)|^2}} \quad &\text{if } \nu_{n+1} >0,
\end{cases}
\end{equation}
for $\nu=(\nu_1,\ldots,\nu_{n+1})\in\sN$. If $u\in\fconvz$, then
\[
\oZ_\varphi(u) = \frac 12 \int_{\sN} \tilde{\varphi}(\nu) \d S_n(K^u,\nu).
\]
\end{lemma}

To get a better understanding of \eqref{eq:varphi_tilde}, we consider the special case when $\varphi=u^*$ with $u\in\fconvz$ (see also \cite[Theorem 13.3]{Rockafellar}).

\begin{proposition}
\label{prop:varphi_u_conj}
If $u\in\fconvz$, then $u^*\in \Cr$ with $\rho_{u^*}=h_{\dom(u)}$. Furthermore, if $\varphi=u^*$ and $\tilde{\varphi}$ is as in \eqref{eq:varphi_tilde}, then
\[
\tilde{\varphi}=h_{K^u}
\]
and
\[
\frac{1}{n+1} \oZ_{\varphi}(u) = \oVb_{n+1}(u).
\]
\end{proposition}
\begin{proof}
If $u\in\fconvz$, then there exists some $C\geq 0$ such that
\[
\ind_{\dom(u)}-C \leq u \leq \ind_{\dom(u)},
\]
and therefore
\[
h_{\dom(u)} \leq u^* \leq h_{\dom(u)} + C
\]
on $\Rn$. Thus,
\[
\rho_{u^*}=h_{\dom(u)}
\]
and, in particular, $u^*\in \Cr$. Now by \eqref{eq:boundary_map} and \eqref{eq:ku_floor} we have
\[
\dom(u)=\proj_H K^u,
\]
where $H=e_{n+1}^\perp$. Thus,
\[
\rho_{u^*}=h_{\dom(u)} = h_{\proj_H K^u} = h_{K^u}\big\vert_H.
\]
Furthermore, since $\sqrt{1+|\gnom(\nu)|^2} = 1/|\nu_{n+1}|$ for $\nu\in\sN_{-}$, it follows from \eqref{eq:ku_floor} and \eqref{eq:kfloor_conj_hk} that
\[
\frac{u^*(\gnom(\nu))}{\sqrt{1+|\gnom(\nu)|^2}} = |\nu_{n+1}|\, h_{K^u}(\gnom(\nu),-1) = h_{K^u}(\nu)
\]
for $\nu\in\sN_{-}$. Thus, if $\varphi=u^*$, then $\tilde{\varphi} = h_{K^u}$. Lemma~\ref{le:varphi_tilde} together with \eqref{eq:mixed_vol_int_s} and \eqref{eq:oVb_V} now shows
\[
\oZ_{\varphi}(u) = \frac 12 \int_{\sN} h_{K^u}(\nu) \d S_n(K^u,\nu) = \frac{n+1}{2} V_{n+1}(K^u) = (n+1)\,\oVb_{n+1}(u).
\]
\end{proof}

Let $\varphi\in\Cr$ and let $\tilde{\varphi}\colon \sN \to \R$ be as in \eqref{eq:varphi_tilde}. For the proof of Theorem~\ref{thm:wulff_ineq_fconvz}, we want to find sufficient conditions such that the Wulff shape $W_{\tilde{\varphi}}$ is well-defined. Furthermore, we want to establish a simple description of $\lfloor W_{\tilde{\varphi}}\rfloor$ in terms of $\varphi$. We start with a simple lemma.

\begin{lemma}
\label{le:epi_psi_conj}
    For every function $\psi\colon \R^n \to \R$ one has 
    \[
        \epi(\psi^*)=\bigcap_{\nu \in \s^n_-} H_\nu^- \left(\frac{\psi(\gnom(\nu))}{\sqrt{1+|\gnom(\nu)|^2}}\right).
    \]
\end{lemma}
\begin{proof}
    First, notice that by definition \eqref{eq:def_conv_conj}, the function $\psi^*$ can be written as the pointwise supremum of the family of affine functions $\ell_y(x)=\langle x, y \rangle - \psi(y)$ with $y \in \Rn$. Therefore,
    \[
    \epi(\psi^*)=\bigcap_{y \in \Rn} \epi(\ell_y).
    \]
    Now for $y\in\Rn$ we have
    \[
    \epi(\ell_y) = \{(x,t)\in\Rn\times\R : \langle (x,t),(y,-1)\rangle\leq \psi(y)\},    
    \]
    which is a closed half-space in $\R^{n+1}$ with outer unit normal
    \[
    \gnom^{-1}(y)=\frac{(y,-1)}{\sqrt{1+|y|^2}}\in\sN_-.
    \]
    Thus,
    \[
    \epi(\ell_y)=\Big\{\xi\in\R^{n+1} : \langle \xi,\gnom^{-1}(y)\rangle \leq \tfrac{\psi(y)}{\sqrt{1+|y|^2}}\Big\} = H_{\gnom^{-1}(y)}^-\Big(\tfrac{\psi(y)}{\sqrt{1+|y|^2}}\Big).
    \]
    Together with the fact that $\gnom\colon \sN_-\to\Rn$ is a bijection, this concludes the proof.
\end{proof}

We need the following consequence of \cite[Proposition 3.5]{Fon}, where we write $\interior(A)$ for the \textit{interior} of $A\subset \Rn$.
\begin{lemma}
\label{le:fon_prop_3.5}
Let $\rho\colon \Rn \to\R$ be a $1$-homogeneous continuous function. If $\rho(x)\geq c |x|$ for every $x\in\Rn$ and some $c>0$, then $o\in\interior(W_{\rho})$ and $\rho^*=\ind_{W_{\rho}}$, where $W_{\rho}$ denotes the Wulff shape associated with the positive function $\rho\vert_{\sn}$.
\end{lemma}

\begin{lemma}
    \label{le:wulff_varphi_conj}
    Let $\varphi\in \Cr$ and let $\tilde{\varphi}$ be as in \eqref{eq:varphi_tilde}. If $\varphi$ is such that 
    \begin{equation}
    \label{eq:lemma_ineq}
    c \leq \frac{\varphi(x)}{\sqrt{1+|x|^2}}
    \end{equation}
    for every $x\in\Rn$ and some $c>0$, then $W_{\tilde{\varphi}}$ is well-defined, has positive volume, and
    \[
    \lfloor W_{\tilde{\varphi}}\rfloor=\varphi^*+\ind_{\{\varphi^*\leq 0\}}.
    \]
    In addition, if
    \begin{equation}
    \label{eq:lemma_ineq2}
    \rho_\varphi \leq \varphi
    \end{equation}
    on $\Rn$, then
    \[
    \lfloor W_{\tilde{\varphi}}\rfloor=\varphi^*
    \]
    and $\dom(\varphi^*)=W_{\rho_{\varphi}}$.
\end{lemma}
\begin{proof}
    Since $\varphi\in \Cr$, it follows from the definition of $\rho_\varphi$ in \eqref{eq:rho_varphi} and \eqref{eq:varphi_tilde} that $\tilde{\varphi}$ is continuous. Moreover, \eqref{eq:lemma_ineq} ensures that $\tilde{\varphi}$ is strictly positive, and thus, $W_{\tilde{\varphi}}$ has positive volume. In addition, $W_{\tilde{\varphi}}$ is symmetric with respect to the hyperplane $H$ by construction, that is, $\refl_H(W_{\tilde{\varphi}})=W_{\tilde{\varphi}}$. From the continuity and symmetry of $\tilde{\varphi}$ together with Lemma~\ref{le:epi_psi_conj} we infer 
    \begin{align}\label{eq:wullf_from_epi}
        W_{\tilde{\varphi}}=\bigcap_{\nu\in\sN_{-}}H^-_{\nu}(\tilde{\varphi}(\nu)) \cap \refl_H\left(\bigcap_{\nu\in\sN_{-}}H^-_{\nu}(\tilde{\varphi}(\nu))\right)= \epi(\varphi^*)\cap \refl_H(\epi(\varphi^*)).
    \end{align}
    In particular, this shows that $(x,t) \in \Rn \times \R$ is contained in $W_{\tilde{\varphi}}$ if and only if $\varphi^*(x)\leq t \leq -\varphi^*(x)$, which is possible if and only if $x\in\{\varphi^*\leq 0\}$.

    Assume now that \eqref{eq:lemma_ineq2} holds. Since $\varphi\in \Cr$, this implies that there exists $C\geq 0$ such that
    \[
    \rho_{\varphi}\leq \varphi \leq \rho_{\varphi}+C,
    \]
    and therefore
    \[
    \rho_{\varphi}^* - C \leq \varphi^* \leq \rho_{\varphi}^*
    \]
    on $\Rn$. By Lemma~\ref{le:fon_prop_3.5} together with \eqref{eq:lemma_ineq} we have $\rho_{\varphi}^*=\ind_{W_{\rho_{\varphi}}}$ and thus
    \[
    \dom(\varphi^*)=\{\varphi^*\leq 0\}=W_{\rho_{\varphi}},
    \]
    which furthermore implies $\lfloor W_{\tilde{\varphi}} \rfloor = \varphi^*$ by the first part of the proof.
\end{proof}

\begin{remark}
As shown in the proof of Lemma~\ref{le:wulff_varphi_conj}, a function $\varphi\in\Cr$ satisfies \eqref{eq:lemma_ineq} and \eqref{eq:lemma_ineq2} if and only if there exists a convex body $K$ with $o\in\interior K$, namely $K=W_{\rho_\varphi}$, such that
\[
0\leq \varphi-h_K \leq C
\]
on $\Rn$ for some $C\geq 0$. Note that a similar condition also appears in the description of\linebreak$\fconvz^*$ (see \eqref{eq:fonvz_conj}) and was also pointed out by Klartag in \cite[Equation (8)]{Klartag_07}.\dssymb
\end{remark}

\subsection{Proof of Theorem~\ref{thm:wulff_ineq_fconvz}}
Let $\varphi\in \Cr$ be such that $c\leq \varphi(x)/\sqrt{1+|x|^2}$, $x\in\Rn$, for some $c>0$. By Lemma~\ref{le:varphi_tilde}, Theorem~\ref{thm:wulff}, \eqref{eq:wulff_alt}, \eqref{eq:ku_floor}, and \eqref{eq:oVb_V} we have
\begin{align}
\oZ_\varphi(u) &= \frac 12 \int_{\sN} \tilde{\varphi}(\nu) \d S_n(K^u)\notag \\
&\geq \left(\frac{V_{n+1}(K^u)}{V_{n+1}(W_{\tilde{\varphi}})} \right)^{\frac{n}{n+1}} \frac 12 \int_{\sN} \tilde{\varphi}(\nu) \d S_n(W_{\tilde{\varphi}},\nu) \label{eq:wulff_proof}\\
&= \left(\frac{\oVb_{n+1}(u)}{\oVb_{n+1}(\lfloor W_{\tilde{\varphi}} \rfloor) } \right)^{\frac{n}{n+1}} \oZ_{\varphi}(\lfloor W_{\tilde{\varphi}} \rfloor)\notag
\end{align}
for every $u\in \fconvz$. Together with Lemma~\ref{le:wulff_varphi_conj}, this gives \eqref{eq:wulff_fct}.

Now suppose that $\oVb_{n+1}(u)>0$. Equality holds in \eqref{eq:wulff_fct} if and only if it holds in \eqref{eq:wulff_proof}. By Theorem \ref{thm:wulff}, this is the case if and only if $K^u$ and $W_{\tilde{\varphi}}$ are homothetic. As $u,u_\varphi \in \fconvz$, the homothety of $K^u$ and $W_{\tilde{\varphi}}$ is equivalent to the homothety of the epi-graphs of $u$ and $u_\varphi$ with the exclusion of vertical translation, since $K^u$ and $W_{\tilde{\varphi}}$ are symmetric with respect to $e_{n+1}^\perp$. Thus, equality holds in \eqref{eq:wulff_fct} if and only if there exist $\lambda>0$ and $x_o \in \Rn$ such that \[
u(x)=\lambda \sq (u_\varphi \square I_{\{x_o\}})(x)=\lambda u_\varphi\left( \frac{x-x_o}{\lambda}\right)
\]
for $x\in\Rn$, concluding the proof.
\qed

\subsection{First Variations}
Let $u,v\in\fconvz$. Anticipating Lemma~\ref{le:infconv_ku} below, it is not hard to show that
\[
K^{u\infconv (t\sq v)} = K^u + t K^v
\]
for $t\geq 0$. Thus, by \eqref{eq:oVb_V}, \eqref{eq:first_variation_volume}, Lemma~\ref{le:varphi_tilde}, and Proposition~\ref{prop:varphi_u_conj},
\begin{equation*}
\frac{\d}{\d t} \bigg\vert_{t=0^+} \oVb_{n+1}(u\infconv(t\sq v)) = \frac 12 \frac{\d}{\d t} \bigg\vert_{t=0^+}  V_{n+1}(K^u + t K^v) = \frac 12 \int_{\sN} h_{K^v}(\nu) \d S_n(K^u,\nu) = \oZ_{v^*}(u).
\end{equation*}
The purpose of this section is to show that an analog relation holds true if we replace $u\infconv (t\sq v)=(u^*+t v^*)^*$ with the perturbation $(u^*+t\varphi)^*$ for certain $\varphi \in \Cr$. We will use the following version of \textit{Alexandrov's variational lemma} (see, for example, \cite[Lemma 7.5.3]{Schneider_CB}).
\begin{lemma}
\label{le:Aleksandrov}
    Let $\eta\in C(\sN)$ and let $K \in \KN$ be such that $o\in\interior(K)$. For sufficiently small $t\geq 0$ the function $h_K+t \eta$ is positive and
    \[
    \frac{\d}{\d t}\bigg\vert_{t=0^+} V_n(W_{h_K+t\eta}) = \int_{\sN} \eta(\nu)\d S_{n}(K, \nu).
    \]
\end{lemma}

\begin{proposition}
\label{prop:variational_formula}
If $u\in\fconvz$ and $\varphi\in\Cr$ are such that $\oVb_{n+1}(u)>0$, then
\[
\frac 12 \frac{\d}{\d t}\bigg\vert_{t=0^+} V_{n+1}(W_{h_{K^u}+t\tilde{\varphi}})=\oZ_\varphi(u),
\]
where $\tilde{\varphi}$ is as in \eqref{eq:varphi_tilde}. If in addition $\rho_{\varphi}\leq \varphi$ on $\Rn$, then
\[
\frac{\d}{\d t}\bigg\vert_{t=0^+} \oVb_{n+1}((u^*+t\varphi)^*) = \oZ_{\varphi}(u).
\]
\end{proposition}
\begin{proof}
Since $\oVb_{n+1}(u)>0$ there exists $x_o\in\interior(\dom(u))$ such that $u(x_o)<0$. 
For $u_{x_o}(x)=u(x-x_o)$, $x\in\Rn$, we have
\[
(u_{x_o}^*+t \varphi)^*(x)=(u^*+\langle x_o,\cdot\rangle + t \varphi)^*(x) = (u^*+t\varphi)^*(x - x_o)
\]
for $x\in\Rn$ and $t\geq 0$. In particular, this shows that $\oVb_{n+1}((u^*+t\varphi)^*) = \oVb_{n+1}((u_{x_o}^*+t\varphi)^*)$. Since also $\oZ_{\varphi}(u)=\oZ_{\varphi}(u_{x_o})$ we may assume without loss of generality that $x_o=0$ and thus $o\in\interior(K^u)$. By Lemma~\ref{le:Aleksandrov} and Lemma~\ref{le:varphi_tilde} we now have
\[
\frac{\d}{\d t}\bigg\vert_{t=0^+} V_{n+1}(W_{h_{K^u}+t\tilde{\varphi}})=\int_{\sN} \tilde{\varphi}(\nu)\d S_n(K^u,\nu)=2\oZ_\varphi(u).
\]
Now assume that $\rho_{\varphi}\leq\varphi$ on $\Rn$. We need to show that
\begin{equation}
\label{eq:v_ovb_wulff_pert_to_show}
V_{n+1}(W_{h_{K^u}+t\tilde{\varphi}})=2 \oVb_{n+1}((u^*+t\varphi)^*)
\end{equation}
for sufficiently small $t\geq 0$. Since $V_{n+1}(K^u)=2\oVb_{n+1}(u)>0$ and since $o\in\interior(K^u)$, there exists $c_1>0$ such that $c_1 B^{n+1}\subseteq K^u$ and thus, by \eqref{eq:ku_floor} and \eqref{eq:kfloor_conj_hk},
\begin{equation}
\label{eq:conj_u_c_1}
u^*(x)=h_K(x,-1)\geq c_1 h_{B^{n+1}}(x,-1)= c_1 \sqrt{1+|x|^2}
\end{equation}
for $x\in\Rn$. Furthermore, since $\rho_{\varphi}$ is $1$-homogeneous and bounded on $\sn$, there exists $c_2\geq 0$ such that
\begin{equation}
\label{eq:varphi_c_2}
\varphi(x)\geq \rho_{\varphi}(x)\geq -c_2 |x| \geq -c_2 \sqrt{1+|x|^2} 
\end{equation}
for $x\in\Rn$. Next, for $t\geq 0$ let $\psi_t=u^*+t\varphi$, which by Proposition~\ref{prop:varphi_u_conj} is an element of $\Cr$. By \eqref{eq:conj_u_c_1} and \eqref{eq:varphi_c_2} we now have
\[
0<c_1-t c_2 \leq \frac{\psi_t(x)}{\sqrt{1+|x|^2}}
\]
for every $x\in\Rn$ and $t\in [0,c_1/c_2)$ (with the convention $c_1/c_2=+\infty$ if $c_2=0$). Furthermore, since $u\leq \ind_{\dom(u)}$, it follows from Proposition~\ref{prop:varphi_u_conj} together with our assumptions on $\varphi$ that for every $t\geq 0$,
\[
\psi_t = u^* + t \varphi \geq 
\rho_{\psi_t} = h_{\dom(u)} + t \rho_{\varphi} \geq \rho_{u^*}+ t \rho_{\varphi} = \rho_{\psi_t}
\]
pointwise on $\Rn$. Thus, for $t\in[0,c_1/c_2)$ we may apply Lemma~\ref{le:wulff_varphi_conj} to $\psi_t$ to obtain
\[
(u^*+t\varphi)^* = \psi_t^* = \lfloor W_{\tilde{\psi_t}} \rfloor.
\]
By Proposition~\ref{prop:varphi_u_conj} we have $\tilde{\psi_t}=h_{K^u} + t \tilde{\varphi}$, which together with \eqref{eq:ku_floor} and \eqref{eq:oVb_V} gives \eqref{eq:v_ovb_wulff_pert_to_show}.
\end{proof}

\section{Mixed Functionals and Shadow Systems}
\label{se:mixed}
\subsection{Mixed Functionals}
We begin this section with a simple but essential result for the further exposition. The proof follows standard arguments, as demonstrated, for example, in \cite[Section 1.6.2]{Schneider_CB}. Recall that 

\begin{lemma}
\label{le:infconv_ku}
For $u_1,u_2\in\fconvz$ and $\lambda_1,\lambda_2\geq 0$, we have
\[
K^{(\lambda_1 \sq u_1) \infconv (\lambda_2 \sq u_2)} = \lambda_1 K^{u_1} + \lambda_2 K^{u_2}.
\]
\end{lemma}
\begin{proof}
For $u\in\fconvz$ and $\lambda>0$ we have $\epi(\lambda \sq u) = \lambda  \epi(u)$ and thus, by \eqref{eq:ref_Ku}, $K^{\lambda \sq u} = \lambda K^u$. For the special case $\lambda=0$ we have $\epi(0 \sq u)=\{(o,t)\in \Rn\times\R : t\geq 0\}$ and thus, if follows again from \eqref{eq:ref_Ku}, that $K^{0\sq u} = \{(o,0)\in\Rn\times \R\} = \{o\in \RN\} = 0 K^u$. In particular, $\lambda \cdot u \in \fconvz$ for every $\lambda \geq 0$.

Now let $u_1,u_2\in\fconvz$. We need to show that $K^{u_1 \infconv u_2} = K^{u_1} + K^{u_2}$. Observe that for $x\in\Rn$ and $t\in\R$ we have $(x,t)\in K^{u_1}+K^{u_2}$ if and only if there exist $(x_1,t_1)\in K^{u_1}$ and $(x_2,t_2)\in K^{u_2}$ such that
\[
u_1(x_1)\leq t_1 \leq -u_1(x_1),\quad u_2(x_2)\leq t_2 \leq -u_2(x_2),
\]
and $x=x_1+x_2$, $t=t_1+t_2$. Note, that this shows that $(x,t)\in K^{u_1}+K^{u_2}$ if and only if $(x,-t)\in K^{u_1}+K^{u_2}$ and therefore $K^{u_1}+K^{u_2}\in \KNH$. Thus, if $x\in \proj_H (K^{u_1}+K^{u_2})$, then
\begin{align*}
\lfloor K^{u_1}+K^{u_2}\rfloor (x) &= \inf\{t: (x,t)\in K^{u_1}+K^{u_2}\}\\
&= \inf\{ t : u_1(x_1)+u_2(x_2)\leq t \leq -(u_1(x_1)+u_2(x_2)), x_1+x_2=x\}\\
&= \inf\{ t: u_1(x_1)+u_2(x_2) \leq t\}\\
&= (u_1 \infconv u_2)(x),
\end{align*}
where we have used that $u_1$ and $u_2$ are non-positive on their domains. The desired statement now follows from \eqref{eq:kfloork}.
\end{proof}

Let $\oVb\colon (\fconvz)^{n+1}\to\R$ be given by
\begin{equation}
\label{eq:def_ovb}
\oVb(u_0,\ldots,u_n) = \frac 12 V(K^{u_0},\ldots,K^{u_n})
\end{equation}
for $u_0,\ldots,u_n\in\fconvz$. From \eqref{eq:oVb_V} and Lemma~\ref{le:infconv_ku} together with \eqref{eq:mixed_vol} we immediately obtain
\begin{equation}
\label{eq:mixed_vol_fconvz}
\oVb_{n+1}\big((\lambda_1 \sq u_1) \infconv \cdots \infconv (\lambda_m \sq u_m)\big) = \sum_{i_0,\ldots,i_n=1}^m \lambda_{i_0}\cdots \lambda_{i_n} \oVb(u_{i_0},\ldots,u_{i_n})
\end{equation}
for $m\in\N$, $u_1,\ldots,u_m\in\fconvz$, and $\lambda_1,\ldots,\lambda_m>0$. In particular,  \eqref{eq:mixed_vol_fconvz} serves as an equivalent (implicit) definition of the functional $\oVb$, presuming that $\oVb$ is symmetric in its entries.

Using \eqref{eq:def_ovb} together with known results for mixed volumes of convex bodies, we immediately retrieve properties and results for the functional $\oVb$. For example, the Alexandrov--Fenchel inequality for convex bodies \eqref{eq:AF} shows
\begin{equation}
\label{eq:af_ovb}
\oVb(u_0,u_1,\ldots,u_n)^2\geq \oVb(u_0,u_0,u_2,\ldots,u_n) \oVb(u_1,u_1,u_2,\ldots,u_n)
\end{equation}
for $u_0,\ldots,u_n\in\fconvz$. As a consequence, for every $1\leq m\leq n+1$, the map
\[
\lambda\mapsto \oVb(u_\lambda[m],u_m,\ldots,u_n)^{\frac 1m}, \quad \lambda\in [0,1],
\]
is concave, where we write
\[
u_\lambda = \big((1-\lambda)\sq u_0\big)\infconv \big(\lambda \sq u_1\big).
\]
In order to give additional meaning to such results, we seek to provide a more direct description of $\oVb(u_0,\ldots,u_n)$.

First, the properties of the mixed area measure on $\K^{n+1}$ (compare \cite[Theorem 5.1.7]{Schneider_CB}) together with \eqref{eq:sa_sn} and Lemma~\ref{le:infconv_ku} defines a symmetric map $\Sa$ on $(\fconvz)^n$ into the space of finite Radon measures on $\sn$ such that
\begin{equation}
\label{eq:mixed_s_s_sN0}
\Sa(u_1,\ldots,u_n;\cdot) = \frac 12 S (K^{u_1},\ldots,K^{u_n},\cdot)\Big\vert_{\sN_0},
\end{equation}
for $u_1,\ldots,u_n\in\fconvz$. Equivalently, this measure is characterized by the relation
\begin{equation}
\label{eq:mixed_meas_fconvz}
\Sa\big((\lambda_1 \sq u_1) \infconv \cdots \infconv (\lambda_m \sq u_m)\big);\cdot) = \sum_{i_1,\ldots,i_n=1}^m \lambda_{i_1}\cdots \lambda_{i_n} \Sa(u_{i_1},\ldots,u_{i_n};\cdot)
\end{equation}
for $m\in\N$, $u_1,\ldots,u_m\in\fconvz$, and $\lambda_1,\ldots,\lambda_m>0$.

Next, we need the following generalization of Lemma~\ref{le:int_ku}, which is a consequence of \cite[Corollary 4.9]{Hug_Mussnig_Ulivelli}.

\begin{lemma}
\label{le:int_ku_mixed}
If $\varphi\colon\Rn\to\R$ is bounded and Borel measurable, then
\[
\int_{\sN_{-}} \tfrac{\varphi(\gnom(\nu))}{\sqrt{1+|\gnom(\nu)|^2}} \d S(K^{u_1},\ldots,K^{u_n},\nu) = \int_{\Rn} \varphi(y) \d\MAp(u_1,\ldots,u_n;y)
\]
for every $u_1,\ldots,u_n\in\fconvz$.
\end{lemma}

We can now give the following description of the functional $\oVb$.

\begin{proposition}
\label{prop:rep_ovb}
For $u_0,\ldots,u_n\in\fconvz$ we have
\begin{multline}
\label{eq:rep_ovb}
\oVb(u_0,\ldots,u_n)\\
= \frac{1}{n+1}\left(\int_{\Rn} u_0^*(x) \d\MAp(u_1,\ldots,u_n;x) + \int_{\sn} h_{\dom(u_0)}(z) \d \Sa (u_1,\ldots,u_n;z)\right).
\end{multline}
\end{proposition}
\begin{proof}
Proposition~\ref{prop:varphi_u_conj} shows that
\[
\frac{u^*(\gnom(\nu))}{\sqrt{1+|\gnom(\nu)|^2}} = h_{K^u}(\nu)
\]
for $u\in\fconvz$ and $\nu\in\sN_{-}$. Furthermore, identifying the hyperplane $H\subset \RN$ with $\Rn$, one infers from \eqref{eq:boundary_map} and \eqref{eq:ku_floor} that
\[
h_{K^u}\big\vert_{H} = h_{\proj_H K^u} = h_{\dom(u)}
\]
for $u\in\fconvz$. Thus, by \eqref{eq:def_ovb}, \eqref{eq:mixed_vol_int_s}, \eqref{eq:mixed_s_s_sN0}, Lemma~\ref{le:int_ku_mixed}, and considering the symmetry of $\KNH$,
\begin{align*}
(n+1)\,\oVb(u_0,\ldots,u_n) &= \frac 12 \int_{\sN} h_{K^{u_0}}(\nu) \d S(K^{u_1},\ldots,K^{u_n},\nu)\\
&= \int_{\sN_{-}} h_{K^{u_0}}(\nu) \d S(K^{u_1},\ldots,K^{u_n},\nu) + \frac 12 \int_{\sN_0} h_{K^{u_0}}(\nu) \d S(K^{u_1},\ldots,K^{u_n},\nu)\\
&= \int_{\Rn} u_0^*(x) \d\MAp(u_1,\ldots,u_n;x) + \int_{\sn} h_{\dom(u_0)}(z) \d \Sa (u_1,\ldots,u_n;z),
\end{align*}
for $u_0,\ldots,u_n\in\fconvz$.
\end{proof}

\begin{remark}
While it directly follows from \eqref{eq:def_ovb} that $\oVb(u_0,\ldots,u_n)$ is symmetric in its entries, the individual integrals on the right side of \eqref{eq:rep_ovb} do not share this property in general, even though the respective measures are symmetric in their entries. To see this recall that the measure $\MAp(u_1,\ldots,u_n;\cdot)$ is invariant under vertical translations of the functions $u_1,\ldots,u_n\in\fconvz$ but, in general, $\int_{\Rn} u_0^*(x)\d\MAp(u_1,\ldots,u_n;x)$ changes when we replace $u_0$ with $u_0+c$ for some $c\in\R$ such that also $u_0+c\in\fconvz$. Thus,
\[
(u_0,\ldots,u_n)\mapsto \int_{\Rn} u_0^*(x) \d\MAp(u_1,\ldots,u_n;x)
\]
is not symmetric on $(\fconvz)^{n+1}$. In contrast to this, $h_{\dom(u_0)}$ is invariant under vertical translations of $u_0$ while the measure $\d \Sa (u_1,\ldots,u_n;z)$ is highly dependent on the vertical position of the functions $u_1,\ldots,u_n\in\fconvz$. However, symmetry of the individual integrals is possible under additional assumptions on the functions $u_0,\ldots,u_n$. Such a situation is presented in Corollary~\ref{cor:af_klartag}.\dssymb
\end{remark}

\begin{lemma}
\label{le:s_vanish}
If $u_1,\ldots,u_n\in\fconvz$ are such that $u_i\vert_{\bd(\dom(u_i))}\equiv 0$ for $1\leq i\leq n$, then $\Sa(u_1,\ldots,u_n;\cdot)\equiv 0$.
\end{lemma}
\begin{proof}
Consider first $u \in \fconvz$. If $u\vert_{\bd(\dom(u))}\equiv 0$, then \eqref{eq:s_pushforward} and \eqref{eq:sa_lower_dim} show that $\Sa(u;\cdot)\equiv 0$. Now if $u_1,\ldots,u_m\in\fconvz$ are such that $u_i\vert_{\bd(\dom(u_i))}\equiv 0$ for $1\leq i\leq m$, then also $(\lambda_1\sq u_1) \infconv \cdots \infconv (\lambda_m\sq u_m)\in\fconvz$ vanishes on the boundary of its domain for $\lambda_1,\ldots,\lambda_m\geq 0$. Therefore, the result follows from \eqref{eq:mixed_meas_fconvz} together with the first part of the proof.
\end{proof}

By \eqref{eq:def_ovb}, \eqref{eq:af_ovb}, Proposition~\ref{prop:rep_ovb}, and Lemma~\ref{le:s_vanish}, we immediately obtain the following inequality for integrals with respect to (conjugate) Monge--Amp\`ere measures, which was essentially shown previously by Klartag \cite[Theorem 1.3]{Klartag_07} (cf.\ Remark~\ref{re:klartag_af} below).

\begin{corollary}
\label{cor:af_klartag}
If $u_0,\ldots,u_n\in\fconvz$ are such that $u_i\vert_{\bd(\dom(u_i))}\equiv 0$ for $0\leq i\leq n$, then
\begin{multline*}
\left(\int_{\Rn} u_0^*(x)\d\MAp(u_1,\ldots,u_n;x)\right)^2\\
\geq
\int_{\Rn} u_0^*(x)\d\MAp(u_1,u_1,u_3,\ldots,u_n;x)
\int_{\Rn} u_0^*(x)\d\MAp(u_2,u_2,u_3,\ldots,u_n;x).
\end{multline*}
In addition, under the above assumptions, the map
\[
(u_0,\ldots,u_n)\mapsto \int_{\Rn} u_0^*(x) \d\MAp(u_1,\ldots,u_n;x)
\]
is symmetric in its entries.
\end{corollary}

\begin{remark}
\label{re:klartag_af}
Let us point out that Klartag's inequality in \cite[Theorem 1.3]{Klartag_07} is formulated in terms of compactly-supported, continuous functions $f_0,\ldots,f_n$ on $\Rn$ that are concave on their supports with additional $C^2$ assumptions on their modified Legendre transforms. Apart from the additional regularity assumptions, the functions $f_i$, $0\leq i\leq n$, play the role of the functions $-u_i$ from Corollary~\ref{cor:af_klartag}. In fact, our proof of Corollary~\ref{cor:af_klartag} is based on the same idea as the proof of \cite[Theorem 1.3]{Klartag_07}. However, using (conjugate) mixed Monge--Amp\`ere measures allows us to efficiently circumvent additional regularity requirements. Finally, it should be noted that Corollary~\ref{cor:af_klartag} can also be obtained from \cite[Theorem 1.3]{Klartag_07} by an approximation argument using continuity properties of the measures $\MAp$ (see, for example, \cite[Theorem 5.1]{Colesanti-Ludwig-Mussnig-7}).\dssymb
\end{remark}

\subsection{Shadow Systems}
\label{se:shadow_sytems}
For $z\in\sN$, recall that the \textit{Steiner symmetral} of a convex body $K\in\KN$ in direction $z$ is given by
\[
S_{z} K = \Big\{x+\lambda z : x\in\proj_{z^\perp} K, -\tfrac 12 V_1\big(K\cap (x+z\R)\big) \leq \lambda \leq \tfrac 12 V_1\big(K\cap (x+z\R)\big)\Big\}.
\]
The set $S_{z} K$ is again a convex body and $V_{n+1}(S_{z} K) = V_{n+1}(K)$. Furthermore, it is possible to find a sequence of iterated Steiner symmetrals that converges to a centered Euclidean ball with the same volume as $K\in\KN$. Notably, it was shown by Klain in \cite[Corollary 5.4]{Klain_2012} that it is possible to choose a sequence of directions granting convergence to a centered Euclidean ball independently of the set $K\in\KN$ one starts from (see also \cite{Bianchi_Gardner_Gronchi_IUMJ22}). Applying this result in $n$-dimensional space shows that we can iterate Steiner symmetrizations with respect to directions chosen only from $\sN_0$, thereby simultaneously symmetrizing the convex bodies $K\cap\{x\in\RN: \langle x,e_{n+1}\rangle = s\}$ with $s\in[-h_K(-e_{n+1}),h_K(e_{n+1})]$. The sequence obtained by this procedure converges to the unique body $\bar{K}\in\KN$ whose sections parallel to $e_{n+1}^\perp$ are centered Euclidean balls such that
\[
V_n\big(\bar{K}\cap \{x\in\RN : \langle x,e_{n+1}\rangle = s\}\big) = V_n\big(K \cap \{x\in\RN : \langle x,e_{n+1}\rangle = s\}\big)
\]
for every $s\in\R$. We summarize this consequence of \cite[Corollary 5.4]{Klain_2012} in the following result.
\begin{proposition}
\label{prop:steiner_sequence}
There exists a sequence $z_j\in \sN_0$, $j\in\N$, such that for every $K\in\KN$ the sequence
\[
S_{z_k} \cdots S_{z_1} K, \quad k\in\N,
\]
converges to $\bar{K}$.
\end{proposition}

A far-reaching generalization of Steiner symmetrization is given by \textit{shadow systems}, introduced by Shephard in \cite{shephard_shadow}. For this we identify $\RN=E=e_{n+2}^\perp\subset \R^{n+2}$ and denote by $\pi_y\colon \R^{n+2}\to E$ the \textit{projection} parallel to the line spanned by $y\in\R^{n+2}$ (which, in general, is not an orthogonal projection). For a convex body $K\in\K^{n+2}$, the shadow system induced by $K$ is now defined as
\[
K(x)=\pi_{(x,1)} K\in\KN,\qquad x\in\RN.
\]
The following result is due to \cite[Section 2]{shephard_shadow} (see also \cite[Theorem 10.4.1]{Schneider_CB}).
\begin{theorem}
\label{thm:shadow_sys_v_convex}
If $x\mapsto K_1(x),\ldots,K_{n+1}(x)$ are shadow systems, then the function
\[
x\mapsto V(K_1(x),\ldots,K_{n+1}(x))
\]
is convex.
\end{theorem}

A consequence of Theorem~\ref{thm:shadow_sys_v_convex} is that for $1\leq m\leq {n+1}$ the expression
\[
V(K[m],K_{m+1},\ldots,K_{n+1})
\]
with $K_{m+1},\ldots,K_{n+1}\in\KN$, is non-increasing under Steiner symmetrization of $K\in\KN$. See also \cite[Section 10.4]{Schneider_CB}.

Considering shadow systems that preserve symmetry with respect to $H=e_{n+1}^\perp\subset \RN$, that is $\refl_H(K(x))=K(x)$ for $x\in\RN$, we immediately obtain the following corollary of \eqref{eq:def_ovb} together with Theorem~\ref{thm:shadow_sys_v_convex}.

\begin{corollary}
\label{cor:shadow_sys_v_floor_k_convex}
If $x\mapsto K_1(x),\ldots,K_{n+1}(x)$ are shadow systems such that $K_i(x)\in\KNH$ for $1\leq i\leq n+1$, then the function
\[
x\mapsto \oVb(\lfloor K_1(x)\rfloor, \ldots, \lfloor K_{n+1}(x)\rfloor)
\]
is convex.
\end{corollary}

A particular instance of shadow systems that preserve symmetry with respect to $H$ is obtained if we consider Steiner symmetrization in a direction $z\in\sN_0$. This can be expressed in the form of \textit{parallel chord movement} for $K\in\KN$. For this, let $f,g\colon \Rn\to\R$ be such that
\[
K=\{x+\lambda z: x\in \proj_{z^\perp} K, g(x)\leq \lambda \leq f(x)\}
\]
and set $h(x)=-(f(x)+g(x))$ for $x\in\proj_{z^\perp} K$. For $t\in [0,1]$ we now set
\begin{equation}
\label{eq:K_z_t}
K_z(t)=\big\{y+ h(\proj_{z^\perp} y) t z : y\in K\big\}
\end{equation}
and observe that $K_z(0)=K$, $K_z(1)=\refl_{z^\perp} K$, and $K_z(\frac 12)=S_z K$ (cf.\ \cite{Rogers_Shephard}). As illustrated in \cite[Section 10.4]{Schneider_CB}, one can find a body $\tilde{K}\in \K^{n+2}$ such that $K_z(t)=\pi_{(-tz,1)} \tilde{K}$ and thus, the statement of Theorem~\ref{thm:shadow_sys_v_convex} holds true for the systems $(K_i)_z(t)$, $1\leq i \leq n+1$. Furthermore, it immediately follows from \eqref{eq:K_z_t} that if $K\in\KNH$ and $z\in\sN_0$, then $K_z(t)\in\KNH$ for every $t\in[0,1]$. Thus, for $u\in\fconvz$ and $z\in\sn=\sN_0$ we can define
\begin{equation}
\label{eq:u_z_t}
u_{z,t}=\lfloor (K^u)_z(t)\rfloor
\end{equation}
for $t\in[0,1]$. By \eqref{eq:ref_Ku} and \eqref{eq:boundary_map}, this is equivalent to
\begin{equation}
\label{eq:u_z_t_lvl_sets}
\{u_{z,t}\leq s\} = (\{u \leq s\})_z(t)
\end{equation}
for $s\geq \min_{x \in \dom(u)} u(x)$, which means that $t\mapsto u_{z,t}$ only acts on the level sets of $u$. In particular, \eqref{eq:u_z_t_lvl_sets} provides a canonical extension of the definition of $u_{z,t}$ to $u\in\fconvs$. Let us also remark that notions of shadow systems on function spaces are not new. See, for example, \cite{Pivovarov_Rebollo}. See also \cite{Hoehner_sym}.

In order to take care of boundary terms, we need the following property.
\begin{lemma}
\label{le:u_z_t_bd}
If $u\in\fconvz$ is such that $u\vert_{\bd(\dom(u))}\equiv 0$, then also $u_{z,t}\vert_{\bd(\dom(u_{z,t}))}\equiv 0$ for every $z\in\sn$ and $t\in[0,1]$.
\end{lemma}
\begin{proof}
Let $u\in\fconvz$ and let $m=\min_{x\in\dom(u)} u(x)$. The condition $u\vert_{\bd(\dom(u))}\equiv 0$ is equivalent to
\[
\{u\leq s_1\} \subset \interior\{u\leq s_2\}
\]
for every $m\leq s_1 < s_2 \leq 0$. By \eqref{eq:u_z_t_lvl_sets} together with the fact that the system defined by \eqref{eq:K_z_t} preserves strict inclusions of sets, we obtain
\[
\{u_{z,t}\leq s_1\} \subset \interior\{u_{z,t}\leq s_2\}
\]
for every $m\leq s_1 < s_2 \leq 0$ and $z\in \sn$, and thus $u_{z,t}\vert_{\bd(\dom(u_{z,t}))}\equiv 0$.
\end{proof}

\begin{theorem}
\label{thm:mixed_int_map_convx}
Let $u_0\in\fconvz$ and let $u_1,\ldots,u_n\in\fconvcd$. If for every $1\leq i\leq n$ there exists $c_i\in\R$ such that $u_i\vert_{\bd(\dom(u_i))}\equiv c_i$, then
\[
t\mapsto \int_{\Rn} (u_0)_{z,t}^*(x)\d\MAp\big((u_1)_{z,t},\ldots,(u_n)_{z,t};x\big),\quad t\in [0,1]
\]
is convex.
\end{theorem}
\begin{proof}
Since building the conjugate mixed Monge--Amp\`ere measure is invariant under vertical translations of the functions and since by \eqref{eq:u_z_t_lvl_sets},
\[
(u-c)_{z,t}=u_{z,t}-c
\]
for every $u\in\fconvcd$, $c\in\R$, $z\in\sn$, and $t\in[0,1]$, we may assume that $u_i\vert_{\bd(\dom(u_i))} \equiv 0$ for $1\leq i\leq n$. Thus, by Proposition~\ref{prop:rep_ovb}, Lemma~\ref{le:s_vanish}, and Lemma~\ref{le:u_z_t_bd}, we have
\[
\oVb\big((u_0)_{z,t},\ldots,(u_n)_{z,t}\big) =  \frac{1}{n+1} \int_{\Rn} (u_0)_{z,t}^*(x)\d\MAp\big((u_1)_{z,t},\ldots,(u_n)_{z,t};x\big)
\]
for $z\in\sn$ and $t\in[0,1]$. The statement now follows from Corollary~\ref{cor:shadow_sys_v_floor_k_convex} together with \eqref{eq:u_z_t}.
\end{proof}

For $u\in\fconvs$ let $s_z u$ denote the Steiner symmetral of $u$ in direction $z\in\sn$, which is given by
\begin{equation}
\label{eq:s_z_u}
\{s_z u\leq s\} = S_z \{u\leq s\}
\end{equation}
for $s\in\R$, or equivalently, $s_z u=u_{z,1/2}$. Furthermore, we write $\bar{u}$ for the symmetric rearrangement of $u$, which is the unique convex function such that
\begin{equation}
\label{eq:bar_u}
V_n\big(\{\bar{u}\leq s\}\big)=V_n\big(\{u\leq s\} \big)
\end{equation}
for $s\in\R$ and such that the level sets of $\bar{u}$ are Euclidean balls centered at the origin. Such symmetrizations are well-known; see, for example, \cite{baernstein,kawohl}.

Observe that equation \eqref{eq:s_z_u} shows that $s_z(u+c)=s_z u +c$ for $u\in\fconvs$, $c\in\R$ and $z\in\sn$. Furthermore, it is straightforward to see that if $u\in\fconvz$, then
\begin{equation}
\label{eq:S_z_Ku}
K^{s_z u} = S_z K^u \quad \text{and} \quad K^{\bar{u}}=\overline{K^{u}}.
\end{equation}
In addition, if $u\vert_{\bd(\dom(u))}\equiv 0$, then Lemma~\ref{le:u_z_t_bd} shows that also $s_z u$ has this property, and it is not hard to see that the same is true for $\bar{u}$.

\begin{proposition}
\label{prop:int_ineq_sz_bar}
If $w,u\in\fconvcd$ are such that $u\vert_{\bd(\dom(u))} = c$ for some $c\in\R$, then
\[
\int_{\dom(u)} w^*(\nabla u(x))<\infty,
\]
\[\int_{\dom(u)} \langle y,\nabla u(x)\rangle \d x = 0,
\]
for every $y\in\Rn$, and
\begin{equation*}
\int_{\dom(u)} w^*(\nabla u(x)) \d x \geq \int_{\dom(s_z(u))} (s_z w)^*(\nabla (s_z u)(x))\d x \geq \int_{\dom(\bar{u})} \bar{w}^*(\nabla \bar{u}(x)) \d x
\end{equation*}
for every $z\in\sn$.
\end{proposition}
\begin{proof}
Throughout the proof let $w,u\in\fconvz$ be such that $u\vert_{\bd(\dom(u))}\equiv c$ for some $c\in\R$. Without loss of generality, we may assume that $c=0$ (cf.\ the proof of Theorem~\ref{thm:mixed_int_map_convx}). Furthermore we can write $w=w_o+M$ with $M=\max_{x \in \dom(u)} w(x)$ and $w_o=w-M\in \fconvz$. We now have $w^*=w_o^*-M$, which together with \eqref{eq:def_ovb}, Proposition~\ref{prop:rep_ovb}, and Lemma~\ref{le:s_vanish} gives
\begin{align}
\begin{split}
\label{eq:int_w_nabla_u}
\int_{\dom(u)} w^*(\nabla u(x)) \d x &= \int_{\dom(u)} w_o^*(\nabla u(x)) \d x - M\, V_n(\dom (u))\\
&= \int_{\Rn} w_o^*(x) \d\MAp(u;x) - M\, V_n(\dom(u))\\
&= \frac{n+1}{2}\,V(K^{w_o},K^u[n])-M\,V_n(\dom(u)).
\end{split}
\end{align}
In particular, this implies that $\int_{\dom(u)} w^*(\nabla u(x)) \d x<\infty$ and that $w\mapsto \int_{\dom(u)} w^*(\nabla u(x))\d x$ is invariant under horizontal translations of $w$ (which corresponds to a translation of $K^{w_o}$ in a direction in $e_{n+1}^\perp$). Thus, if for $y\in\Rn$ we set $w_y(x)=w(x-y)$, $x\in\Rn$, then
\[
\int_{\dom(u)}w^*(\nabla u(x)) \d x=\int_{\dom(u)} (w_y)^*(\nabla u(x))\d x = \int_{\dom(u)} \big(w^*(\nabla u(x)) + \langle y,\nabla u(x)\rangle \big) \d x,
\]
and therefore $\int_{\dom(u)} \langle y,\nabla u(x)\rangle \d x = 0$. Moreover, Proposition~\ref{prop:steiner_sequence} and Theorem~\ref{thm:shadow_sys_v_convex}, together with the continuity of mixed volumes show
\[
V(K^{w_o},K^u[n])\geq V(S_z K^{w_o},S_z K^u[n]) \geq V(\bar{K^{w_o}}, \bar{K^u}[n])
\]
for every $z\in \sN_0=\sn$. In addition, by \eqref{eq:s_z_u} and \eqref{eq:bar_u},
\[
V_n(\dom(u)) = V_n(\dom(s_z u)) = V_n(\dom{\bar{u}}).
\]
Together with \eqref{eq:S_z_Ku} and another application of \eqref{eq:int_w_nabla_u}, this shows the desired statement.
\end{proof}

\subsection{Proof of Theorem~\ref{thm:int_w_nabla_u}}
Let $w\in\fconvs$ and $u\in\fconvcd$ be as in the statement of the theorem. Since $w$ is proper, there exists $x_o\in\Rn$ such that $w(x_o)<\infty$. For $w_{x_o}(x)=w(w-x_o)$, $x\in\Rn$, it follows from Proposition~\ref{prop:int_ineq_sz_bar} and \eqref{eq:cond_int_w_nabla_u} that
\[
\int_{\dom(u)} (w_{x_o})^*(\nabla u(x)) \d x = \int_{\dom(u)} \big( w^*(\nabla u(x)) + \langle x_o,x \rangle \big) \d x = \int_{\dom(u)} w^*(\nabla u(x)) \d x.
\]
Thus, without loss of generality, we may assume that $x_o=o$. Furthermore, since for any $a\in\R$, we have
\[
\int_{\dom(u)} (w-a)^*(\nabla u(x))\d x = \int_{\dom(u)} w^*(\nabla u(x))\d x + a\, V_n(\dom(u)),
\]
we may assume that $w(o)\leq 0$. This is equivalent to $w\leq \ind_{\{o\}}$ and therefore $w^*\geq 0$. For $j\in\N$ let
\begin{equation}
\label{eq:def_w_j}
w_j = w+\ind_{\{w\leq j\}}
\end{equation}
and note that due to the convergence of the corresponding level sets the sequence $w_j$, $j\in\N$, epi-converges to $w$ as $j\to\infty$. Furthermore, $w_j\in\fconvcd\subset\fconvs$ with $w_j(o)\leq 0$ for $j\in\N$. Moreover, $w_j\geq w$ pointwise. Thus it follows from the continuity of convex conjugation and the equivalent description of epi-convergence on $\fconvf$ that $(w_j)^*$, $j\in\N$, is a sequence of non-negative convex functions in $\fconvf$ that converges pointwise to $w^*$ such that $0\leq (w_j)^*\leq w^*$ for every $j\in\N$.

Next, observe that $w(o)\leq 0$ is equivalent to $0\in \{w\leq 0\}$. Thus, it follows from \eqref{eq:s_z_u} that for $z\in\sn$ also $(s_z w)(o)\leq 0$ and similarly, by \eqref{eq:bar_u}, that $\bar{w}(o)\leq 0$. Thus, we may repeat the considerations above for these functions and further note that \eqref{eq:s_z_u} and \eqref{eq:def_w_j} show that $(s_z w)_j=s_z w_j$ and similarly $(\bar{w})_j = \bar{(w_j)}$ for $j\in\N$. By Proposition~\ref{prop:int_ineq_sz_bar} we now have
\begin{equation}
\label{eq:int_wj_nabla_u}
\int_{\dom(u)} (w_j)^*(\nabla u(x)) \d x \geq \int_{\dom(s_z (u))} (s_z w_j)^*(\nabla(s_z u)(x)) \d x \geq \int_{\dom(\bar{u})} (\bar{w}_j)^*(\nabla \bar{u}(x)) \d x
\end{equation}
for every $z\in\sn$ and $j\in\N$. Furthermore, \eqref{eq:cond_int_w_nabla_u} and the dominated convergence theorem show that
\begin{align*}
\lim_{j\to\infty} \int_{\dom(u)} (w_j)^*(\nabla u(x)) \d x  &= \lim_{j\to\infty} \int_{\Rn} (w_j)^*(x) \d\MAp(u;x)\\
&= \int_{\Rn} w^*(x) \d\MAp(u;x)\\
&= \int_{\dom(u)} w^*(\nabla u(x)) \d x.
\end{align*}
Since corresponding statements also hold for the sequences $(s_z w_j)^*$, $j\in\N$, and $(\bar{w}_j)^*$, $j\in\N$, we obtain \eqref{eq:int_ineq_sz_bar} after taking limits in \eqref{eq:int_wj_nabla_u}.
\qed

\subsection*{Acknowledgments}
The first author was supported by the Austrian Science Fund (FWF):\ 10.55776/J4490 and\linebreak\mbox{10.55776/P36210}. The second author was supported by the Austrian Science Fund (FWF):\ 10.55776/P34446 and, in part, by the Gruppo Nazionale per l’Analisi Matematica, la Probabilit\'a e le loro Applicazioni (GNAMPA) of the Istituto Nazionale di Alta Matematica (INdAM).


\vfill

\parbox[t]{8.5cm}{
Fabian Mussnig\\
Institut f\"ur Diskrete Mathematik und Geometrie\\
TU Wien\\
Wiedner Hauptstra{\ss}e 8-10/1046\\
1040 Wien, Austria\\
e-mail: fabian.mussnig@tuwien.ac.at}

\bigskip

\parbox[t]{8.5cm}{
Jacopo Ulivelli\\
Institut f\"ur Diskrete Mathematik und Geometrie\\
TU Wien\\
Wiedner Hauptstra{\ss}e 8-10/1046\\
1040 Wien, Austria\\
e-mail: jacopo.ulivelli@tuwien.ac.at}

\end{document}